\documentclass{article}
\usepackage[top=4cm, bottom=4cm, left=3cm, right=3cm]{geometry}
\usepackage{graphics}
 \usepackage{graphicx}
 \usepackage{epsfig}
 \usepackage{amsmath}
\usepackage{amsfonts}
\usepackage{amssymb}
\usepackage{xcolor}
\usepackage{enumerate}
\usepackage{hyperref}
\hypersetup{colorlinks=true,linkcolor=blue,filecolor=mangeta}

\newtheorem{theorem}{Theorem}[section]

\newtheorem{prop}{Proposition}[section]
\newtheorem{leme}{Lemma}[section]
\newtheorem{dfnt}{Definition}[section]
\newtheorem{remark}{Remark}

\newtheorem{assumption}{Assumption}[section]

\newenvironment{proof}[1][Proof]{\textbf{#1.} }{\ \rule{0.5em}{0.5em}}

\numberwithin{equation}{section}
\def \dis {\displaystyle}

\def \R {\mathbb{R}}

\def \O {\mathcal{O}}

\def\dT{dx\, dt}

\def \hvarphi \widehat{\varphi}

\def\dq{dx\,dt}

\def \hvarphi \widehat{\varphi}

\begin{document}
  \title{Hierarchical null controllability of a semilinear degenerate parabolic equation with a gradient term}
 \author{{{Landry Djomegne} \thanks{{\it
 				University of Dschang, BP 67 Dschang, Cameroon, West region,
 				email~: {\sf landry.djomegne\char64yahoo.fr} }}}
 	{{ \quad Cyrille Kenne} \thanks{{\it Laboratoire LAMIA, Universit\'e des Antilles, Campus Fouillole, 97159 Pointe-{\`a}-Pitre Guadeloupe (FWI)-Laboratoire L3MA, UFR STE et IUT, Universit\'e des Antilles, Schoelcher, Martinique,
 				email~: {\sf kenne853\char64gmail.com} }}}
 	{{ \quad Ren\'e Dorville} \thanks{{\it
 				Laboratoire L3MA, UFR STE et IUT, Universit\'e des Antilles, Schoelcher, Martinique,
 				email~: {\sf rene.dorville\char64orange.fr} }}}
 			{{ \quad Pascal Zongo} \thanks{{\it
 						Laboratoire L3MA, UFR STE et IUT, Universit\'e des Antilles, Schoelcher, Martinique,
 						email~: {\sf pascal.zongo\char64gmail.com} }}}
 }

  \date{\today}
  
  \maketitle	
  	

\begin{abstract}
In this paper, we apply the hierarchical strategy to a semilinear weakly degenerate parabolic equation involving a gradient term. We use the Stackelberg-Nash strategy with one leader which tries to drive the solution to zero and two followers intended to solve a Nash equilibrium corresponding to a bi-objective optimal control problem. Since the system is semilinear, the functionals are not convex in general. To overcome this difficulty, we first prove  the existence and uniqueness of the Nash quasi-equilibrium, which is a weaker formulation of the Nash equilibrium. Next, with additional conditions, we establish the equivalence between the Nash quasi-equilibrium and the Nash equilibrium. We establish a suitable Carleman inequality for the adjoint system and then an observability inequality. Based on this observability inequality, we prove the null controllability of the linearized system. Then, due to the Kakutani's fixed point Theorem, we obtain the null controllability of the main system.
\end{abstract}

\textbf {2010} Mathematics Subject Classification. {35K55; 35K65; 91A65; 93B05; 93C20.}\par

\section{Introduction}\paragraph{}

Let $\Omega=(0,1)\subset\R$.  Let $\omega$, $\omega_1$ and $\omega_2$ be three nonempty subsets of $\Omega$ with $\omega_i\cap\omega=\emptyset,\ i=1,2$.  For the  real number $T > 0$, we denote $Q =(0, T )\times \Omega$,
$\omega_T =(0, T )\times\omega $, $\omega_{1,T} =(0, T )\times\omega_1 $ and $\omega_{2,T} =(0, T )\times\omega_2 $. We are interested in the following degenerate semilinear parabolic equation:

  \begin{equation}\label{eq}
  \left\{
  \begin{array}{rllll}
  \dis y_{t}-\left(a(x)y_{x}\right)_{x}+F(y,y_x)  &=&h\chi_{\omega}+v^1\chi_{\omega_1}+v^2\chi_{\omega_2}& \mbox{in}& Q,\\
  \dis y(t,0)=y(t,1)&=&0& \mbox{on}& (0,T), \\
  \dis  y(0,\cdot)&=&y^0 &\mbox{in}&\Omega,
  \end{array}
  \right.
  \end{equation}
  where $y=y(t,x)$ denotes the state, $v^i=v^i(t,x),\ i=1,2$ and $h=h(t,x)$ are different control functions  whose act on the system through the subsets $\omega_{i}$ and $\omega$ respectively. These functions $v^i$ and $h$ represent respectively the followers and leader controls. Here, $\chi_{\mathcal{A}}$ denotes the characteristic function of the set $\mathcal{A}$. The function $y^0\in L^2(\Omega)$ is the initial data.  We denote by $y_{t}:=\frac{\partial y}{\partial t}$ and $y_{x}:=\frac{\partial y}{\partial x}$, the partial derivative of $y$ with respect to $t$ and $x$ respectively.
  
  The function $a:=a(\cdot)$ represents the dispersion coefficient and we assume that it satisfies the following hypothesis
  \begin{equation}\label{k}
  	\begin{array}{llll}
  		\dis a\in \mathcal{C}([0,1])\cap\mathcal{C}^1((0,1]),\ \ a>0\ \mbox{in}\ (0,1] \ \mbox{and}\ a(0)=0,\\
  		\dis \exists \tau\in [0,1)\ :\ xa^\prime(x)\leq \tau a(x),\ x\in [0,1].
  	\end{array}
  \end{equation}
  For further need, we consider as in \cite{carmelo2020} a function $\beta(x)$ such that
  \begin{equation}\label{b}
  	\frac{\beta(x)}{x}\in L^\infty(\Omega).
  \end{equation}

  \begin{assumption}\label{assum}$ $
  	 The function $F:\mathbb{R}\times\mathbb{R}\rightarrow \mathbb{R}$ satisfies the following:
  \begin{itemize}
  	\item[($H_1$)] $F(0,0)=0$,
  	\item[($H_2$)]$ F\in \mathcal{C}^2(\mathbb{R}\times\mathbb{R}),$
  	\item[($H_3$)]there exists $K>0:\ |F(s,p)-F(s^\prime,p^\prime)|\leq K\left[|s-s^\prime|+|p-p^\prime|\right],\ \ \forall s,s^\prime,p,p^\prime\in \mathbb{R}$, i.e. $F$ is a globally Lipschitz function, 
  	\item[($H_4$)]$F(s,p)=F_1(s,p)s+F_2(s,p)p$, for every $(s,p)\in \mathbb{R}^2$,
  	\item[($H_5$)] there exists $\dis M_1>0:\ |F_1(s,p)|+\frac{1}{\beta(x)}|F_2(s,p)|\leq M_1,\ \forall(s,p)\in \mathbb{R}^2$,
  	\item[($H_6$)]$\dis \exists M_2>0:\ \sum_{i=1}^{2}|D_iF(s,p)|+\sum_{i,j=1}^{2}|D_{ij}^2F(s,p)|\leq M_2,\ \forall (s,p)\in \mathbb{R}^2$,
  \end{itemize}
where $D_iF$ and $D_{ij}^2F$ are respectively the first and second order partial derivatives of $F$.
  \end{assumption}
 
 \begin{remark}\label{rmq}$ $
 \begin{enumerate}[(a)]
 \item From \eqref{k}, it follows that the function $x\mapsto\dis \frac{x^2}{a(x)}$ is non-decreasing on $(0,1]$.
 
 \item 	Using the point \textit{(a)} and the assumption \eqref{b} on $\beta(x)$, we observe that
 \begin{equation}\label{bet}
 \frac{\beta^2(x)}{a(x)}	\leq C\frac{x^2}{a(x)}\leq \frac{C}{a(1)},\;\;\;\; \text{for a.e}\;\; x\in (0,1)\ \mbox{where},\ C>0.
 \end{equation}
 Then, we deduce that for almost every $x\in (0,1)$,
 \begin{equation}\label{b1}
 |\beta(x)|\leq L\sqrt{a(x)},\ \mbox{where}\ L>0.
 \end{equation}
 These results will be useful in order to prove the existence of the solution for the linearized problem \eqref{ywell}.
 \end{enumerate}	
 
 \end{remark}
   
\begin{remark}$ $

	The system \eqref{eq} can models the dispersion of a gene in a given population. In this case, $x$ represents the gene type and $y:=y(t,x)$ is the distribution of individuals at time $t$ and of gene type $x$ of the population. Genetically speaking, such a property of degeneracy is natural since it means that if each population is not of gene type, it cannot be transmitted to its offspring \cite{younes2014}.
\end{remark}

\par

In this paper, we apply the Stackelberg-Nash strategy introduced in \cite{Von1934Stackelberg} to the system \eqref{eq}. 
It has been applied for the first time in the context of the control of evolution equations by J. L. Lions in \cite{Lions1994Stackelberg1, Lions1994Stackelberg2}. In the last years, other authors have used hierarchical control in the sense of Lions, see for instance \cite{mercan1, mercan2, djomegnebackward, teresa2018, liliana2020, romario2018, djomegne2018, djomegnelinear}. 

 In  \cite{araruna2015anash}, the authors developed the first hierarchical results in the context of exact controllability framework for a class of parabolic equations (linear and semi-linear). These results were improved in \cite{araruna2017} by imposing some weak conditions on observations domains for the followers. This same idea were also applied for the wave equation in \cite{ararunawave2018}. The authors in \cite{araruna2019boundary} dealt with the Stackelberg-Nash exact controllability of parabolic equations with distributed and boundary controls. In the context of hierarchical strategy for coupled systems, K\'er\'e et {\it al.} \cite{kere2017coupled}, considered a bi-objective control strategy for a coupled parabolic equations with a finite constraints  on one of the states. The authors in \cite{teresacoupled} and \cite{teresa} studied a Stackelberg-Nash strategy for a cascade system of parabolic equations. 
 
 More recently in \cite{danynina2021}, J. Limaco et {\it al.} applied the Stackelberg-Nash strategy to a coupled  quasi-linear parabolic system with controls acting in the interior on the domain. N. Carre\~{n}o and M. C. Santos in \cite{santos2019} applied the Stackelberg-Nash strategy to the Kuramoto-Sivashinsky equation with a distributed leader, and two followers. In \cite{djomegne2021}, the author studied the Stackelberg-Nash strategy for a non linear parabolic equation in an unbounded domain. In \cite{dany2021}, D. N. Huaman applied the Stackelberg-Nash strategy to control a quasi-linear parabolic equations in dimensions $1$, $2$ or $3$. Recently, in \cite{omar2021}, the authors applied the hierarchical control to the anisotropic heat equation with dynamic boundary conditions and drift terms. 

In all the above cited works, the hierarchic strategy were applied to non degenerate systems. To the best of our knowledge, the only work dealing with hierarchical strategy applied to a degenerate equation is the one in \cite{araruna2018stackelberg}, were the authors studied the Stackelberg-Nash strategy for both linear and semilinear degenerate parabolic equations. Here, the non linearities involve the first derivative. This inclusion of a first-order term in the expression of $F$ in \eqref{eq} has not been addressed before to our best knowledge. 
  
In this paper we are interested in the Stackelberg-Nash null controllability strategy for the degenerate problem with a gradient term \eqref{eq}. To be more specific, for $i=1,2$, we introduce the non-empty open sets $\omega_{i,d}\subset \Omega$, representing the observation domains of the followers, and the fixed target functions $y_{i,d}\in L^2((0,T); \omega_{i,d})$. Let us define the following cost functional
  \begin{equation}\label{all16}
  J_i(h;v^1,v^2)=\frac{\alpha_i}{2}\int_0^T\int_{\omega_{i,d}}|y-y_{i,d}|^2\ dxdt+\frac{\mu_i}{2}\int_0^T\int_{\omega_i}|v^i|^2\ dxdt,\ i=1,2,
  \end{equation}
  where  $\alpha_i$ and $\mu_i$ are two positive constants.

   We aim to choose the controls $v^i$ and $h$ in order to achieve two different objectives:
   \begin{itemize}
   	\item The main goal is to choose $h$ such that the following null controllability objective holds:
   	\begin{equation}\label{obj1}
   	y(T,\cdot;h;v^1,v^2)=0\ \mbox{in}\ \Omega.
   	\end{equation}
   	\item The second goal is the following: given the functions $y_{i,d}$ and $h$, we want to choose the control $v^i$  minimizing $J_i$ given by \eqref{all16}. This means that, throughout the interval $(0,T)$,  the control $v^i$ will be chosen such that:
   	\begin{equation}\label{obj2}
   	\begin{array}{rll}
   	&&\mbox{the solution}\ y(t,x;h;v^1,v^2)\ \mbox{of}\ \eqref{eq}\ \mbox{remains "not too far" from a desired target}\  y_{i,d}(t,x) \\ 
   	&&\mbox{in the observability domain}\  \omega_{i,d},\ i=1,2.
   	\end{array}
   	\end{equation}
   \end{itemize} 
Our goal is to prove that, for any initial data $y^0\in L^2(\Omega)$, there exist a control $h\in L^2(\omega_T)$ (called leader) and an associated Nash equilibrium $(\hat{v}^1,\hat{v}^2)^{t}=(\hat{v}^1(h),\hat{v}^2(h))^{t}\in \mathcal{H}=L^2((0,T);L^2(\omega_1))\times L^2((0,T);L^2(\omega_2))$ (called followers) such that the associated state $y$ of system \eqref{eq} satisfies \eqref{obj1}. To do this, we follow the Stackelberg-Nash strategy described as follows:

\begin{enumerate}
\item For each choice of the leader $h$, we look for a Nash equilibrium pair for the costs $J_i,\ i=1,2$ given by \eqref{all16}. That is, find the controls $(\hat{v}^1,\hat{v}^2)^{t}=(\hat{v}^1(h),\hat{v}^2(h))^{t}\in \mathcal{H}$ satisfying 
 \begin{equation*}\label{Nash}
	\begin{array}{rllll}
	\dis J_1(h;\hat{v}^1,\hat{v}^2)\leq J_1(h;v^1,\hat{v}^2),\ \ \ \forall v^1\in L^2((0,T);L^2(\omega_1)),\\
		\dis J_2(h;\hat{v}^1,\hat{v}^2)\leq J_2(h;\hat{v}^1,v^2),\ \ \ \forall v^2\in L^2((0,T);L^2(\omega_2)),
	\end{array}
\end{equation*}
or equivalently
 \begin{equation}\label{Nash1}
	\begin{array}{rllll}
		\dis	J_1(h;\hat{v}^1,\hat{v}^2)=\min_{v^1\in L^2((0,T);L^2(\omega_1))}J_1(h;v^1,\hat{v}^2),\\
		\dis	J_2(h;\hat{v}^1,\hat{v}^2)=\min_{v^2\in L^2((0,T);L^2(\omega_2))}J_2(h;\hat{v}^1,v^2).
	\end{array}
\end{equation}

\item Once the Nash equilibrium has been identified and fixed for each $h$, we look for a control $\bar{h}$ such that
\begin{equation}\label{mainobj}
	y(T,\cdot;\bar{h};\hat{v}^1(\bar{h}),\hat{v}^2(\bar{h}))=0\ \mbox{in}\ \Omega.
\end{equation}
		
\end{enumerate}

\begin{remark}
	$ $
	\begin{enumerate}[(a)]
	\item In the linear case, the functionals $J_i,\ i=1,2$ are differentiable and convex and the pair $(\hat{v}^1,\hat{v}^2)$ is a Nash equilibrium for $(J_1,J_2)$ if and only if
	\begin{equation}\label{Nashprime}
	\begin{array}{rllll}
	\dis	\frac{\partial J_1}{\partial v^1}(h;\hat{v}^1,\hat{v}^2)(v^1,0)=0,\ \ \forall v^1\in L^2((0,T);L^2(\omega_1)),\  \ \hat{v}^i\in L^2((0,T);L^2(\omega_i)),\\
	\dis	\frac{\partial J_2}{\partial v^2}(h;\hat{v}^1,\hat{v}^2)(0,v^2)=0,\ \ \forall v^2\in L^2((0,T);L^2(\omega_2)),\  \ \hat{v}^i\in L^2((0,T);L^2(\omega_i)).
	\end{array}
	\end{equation}
	\item 	In the semi-linear framework, the corresponding functionals $J_1$ and $J_2$ are not convex in general. For this reason, we must consider the following weaker definition of Nash equilibrium.	
	
	\end{enumerate}
	
\end{remark}

 \begin{dfnt}\label{dfntquasi}$ $
 	
 	Let the leader control $h$ be given. The pair ($\hat{v}^1,\hat{v}^2)$ is called a Nash quasi-equilibrium of functionals $(J_1,J_2)$ if the condition \eqref{Nashprime} is satisfied.		
 \end{dfnt}

In the case where only a leader control is exerted on $\omega$, i.e. $v^1\equiv v^2$, the null controllability of the following degenerate diffusion equation, has been already investigated
\begin{eqnarray}\label{}
	\dis W_{t}-\left(x^\alpha W_{x}\right)_{x}+k(t,x)W=h(t,x)\chi_{\omega},\ (t,x)\in Q,\label{eqa}\\
		\dis W(t,0)=0\ \mbox{if}\ 0<\alpha<1,\label{eqb}\\
		\dis  (x^\alpha W)(t,0)=0\ \mbox{if}\ \alpha\geq 1,\ t\in(0,T), \label{tea}\\
		\dis W(t,1)=0,\ \ t\in (0,T),\label{eqc}\\
		\dis  W(0,\cdot)=W_0(x),\ x\in \Omega,\label{eqd}	
\end{eqnarray}
where $\alpha>0$ and $k\in L^\infty(Q)$. It was shown that the system \eqref{eqa}-\eqref{eqd} is null controllable if $ 0<\alpha<2$ \cite{cannarsa2005, cannarsa2008, alabau2006}, while not null controllable if $\alpha\geq 2$ \cite{cannarsa2004}. Note that if $ 0<\alpha<1$, the degeneracy of \eqref{eqa} is weak, while strong if $\alpha\geq 1$. In \cite{carmelo2010}, Flores and Teresa considered the degenerate convection-diffusion equation
\begin{equation}\label{te}
y_t-(x^\alpha y_x)_x+x^{\alpha/2}b(t,x)y_x+k(t,x)y=h(t,x)\chi_{\omega},\ (t,x)\in Q,	
\end{equation}
with $b\in L^\infty(Q)$ and proved that the system \eqref{te} and \eqref{eqb}-\eqref{eqd} is null controllable for $ 0<\alpha<2$. In \eqref{te}, the gradient term depends on $\alpha$ and we said that the convection term can be controlled by the diffusion term. In \cite{carmelo2020}, the authors presented a null controllability result for the system \eqref{eq} and \eqref{tea}. The main result has been obtained after the proof of a new Carleman inequality for a degenerate linear parabolic equation with first order terms. F. Xu et \textit{et al.} in \cite{xu2020} proved the null controllability of the semilinear degenerate system \eqref{eq} and \eqref{tea} with $a(x)=x^\alpha,\ 0<\alpha<2$. The authors in \cite{wang2014}, studied the degenerate convection-diffusion equation with the convection term
independent of $\alpha$
\begin{equation}\label{teb}
	y_t-(x^\alpha y_x)_x+(b(t,x)y)_x+k(t,x)y=h(t,x)\chi_{\omega},\ (t,x)\in Q,	
\end{equation}
and proved that the system \eqref{teb} and \eqref{eqb}-\eqref{eqd} is null controllable for $ 0<\alpha<1/2$. The authors in \cite{wang2018}  considered  the degenerate equation
\begin{equation}\label{tec}
	y_t-(x^\alpha y_x)_x+b(t,x)y_x+k(t,x)y=h(t,x)\chi_{\omega},\ (t,x)\in Q,	
\end{equation}
and showed that the system \eqref{tec} and \eqref{eqb}-\eqref{eqd} is null controllable for $ 0<\alpha<1$ if $b$, $b_x$, $b_{xx}$, $b_t\in L^\infty(Q)$, i.e. $b\in W^{2,1}_{\infty}(Q)$.

\subsection{Main results}
The next result establish the equivalence between the Nash quasi-equilibrium and the Nash equilibrium under some additional conditions.
More precisely, we have the following result.

 \begin{prop}\label{nash}$ $
 	
	Let us assume that $a(\cdot)$ satisfies \eqref{k}, $y_{i,d}\in L^\infty((0,T); \omega_{i,d})$ and Assumption \ref{assum} are satisfied. Assume that $h\in L^2(\omega_T)$ and $\mu_i,\ i=1,2$ are sufficiently large. 
	Then, if $(\hat{v}^1,\hat{v}^2)$ is a Nash quasi-equilibrium for $J_i,\ i=1,2$, there exists a constant $C>0$ independent of $\mu_i,\ i=1,2$ such that 
	\begin{equation}
		\dis D_i^2J_i(h;\hat{v}^1,\hat{v}^2)\cdot(w^i,w^i)\geq C \|w^i\|^2_{L^2((0,T);L^2(\omega_i))},\ \forall w^i\in L^2((0,T);L^2(\omega_i)),\ i=1,2.
	\end{equation}
	In particular, the functionals $(J_1,J_2)$ are convex in $(\hat{v}^1, \hat{v}^2)$ and therefore the pair $(\hat{v}^1, \hat{v}^2)$ is a Nash equilibrium for $J_i,\ i=1,2$ of the system \eqref{eq}.
	
\end{prop}

 To state the main result of this paper, we impose the following assumptions:
 \begin{equation}\label{od}
 	\dis \omega_{1,d}=\omega_{2,d}:=\omega_d
 \end{equation}
and 
 \begin{equation}\label{oda}
 	\dis \omega_d\cap\omega\neq\emptyset.
 \end{equation}

The main result of this paper is the following:

\begin{theorem}\label{theolinear}$ $
	
Let $T > 0$ and $y_0\in L^2(\Omega)$. Suppose that \eqref{oda}-\eqref{od} holds and $\mu_i,\ i=1,2$ are large enough, $a(\cdot)$ satisfies \eqref{k} and the Assumption \ref{assum} are satisfied.
Then,
	there exists a positive real weight function $\kappa=\kappa(t)$ blowing up at $t=T$ such that for any $y_{i,d}\in L^2((0,T); \omega_{i,d})$ satisfying
	\begin{equation}\label{nou}
		\int_{0}^{T}\int_{\omega_d}\kappa^{-2}|y_{i,d}|^2\ dxdt<+\infty,\ \ i=1,2,
	\end{equation}
	 there exist a control $\bar{h}\in L^2(\omega_T)$ and an associated Nash equilibrium $(\hat{v}^1, \hat{v}^2)\in \mathcal{H}=L^2((0,T);L^2(\omega_1))\times L^2((0,T);L^2(\omega_2))$ such that the corresponding solution to \eqref{eq} satisfies \eqref{mainobj}. 
\end{theorem}

 The rest of the paper is organized as follows. In Section \ref{well}, we introduce some weighted function spaces needed to establish some well posedness results. Section \ref{low} deals with the proof of existence and characterization of Nash-quasi equilibrium. Under some conditions, we prove equivalence between Nash equilibrium and Nash-quasi equilibrium in Section \ref{conv}. In Section \ref{Carleman}, we prove some suitable Carleman inequalities and we deduce the null controllability result in Section \ref{null}. A conclusion is given in Section \ref{conclusion}.

 \section{Well-posedness result}\label{well}
 
 Throughout this paper, the usual norm in $L^\infty(Q)$ will be denoted by $\|\cdot\|_{\infty}$.
 In this section, we consider the following linear problem
 \begin{equation}\label{ywell}
 	\left\{
 	\begin{array}{rllll}
 		\dis y_{t}-\left(a(x)y_{x}\right)_{x}+a_0y+\beta(x)b_0y_x  &=&\dis h\chi_{\omega}+v^1\chi_{\omega_1}+v^2\chi_{\omega_2}& \mbox{in}& Q,\\
 		\dis y(t,0)=y(t,1)&=&0& \mbox{on}& (0,T), \\
 		\dis y(0,\cdot)&=&y^0&\mbox{in}&\Omega,
 	\end{array}
 	\right.
 \end{equation}
 where $a_0,b_0\in L^\infty(Q)$, $y^0\in L^2(\Omega)$ and $h\chi_{\omega}, v^1\chi_{\omega_1}, v^2\chi_{\omega_2}\in L^2(Q)$.
 
 The following assumption will help us to prove the existence result of problem \eqref{ywell}.
 \begin{assumption}$ $\label{aspt}
 	
 Let the constant $L$ be defined by \eqref{b1}. Then there exists a constant $\dis \alpha>0$ such that
 	$$
 	a_0(t,x)\geq \alpha\ \mbox{for all}\ (t,x)\in Q \ \mbox{and}\ \dis \alpha>\frac{1}{2}L^2\|b_0\|^2_{\infty}.
 	$$ 	
 \end{assumption}

 In order to prove the well-posedness of the degenerate system \eqref{ywell}, we introduce the following weighted Hilbert spaces (in the sequel, abs. cont. means absolutely continuous):
  
  \begin{equation*}
  	 H^1_{a}(\Omega)=\{u\in L^2(\Omega): u\ \mbox{is abs. cont. in}\ [0,1]: \sqrt{a}u_x\in L^2(\Omega),\ u(0)=u(1)=0\},
  \end{equation*}
 endowed with the norm  
  $$
  \|u\|^2_{H^1_{a}(\Omega)}=\|u\|^2_{L^2(\Omega)}+\|\sqrt{a}u_x\|^2_{L^2(\Omega)},\ \ \ u\in H^1_{a}(\Omega)
  $$
  and
 \begin{equation*}\label{}
 \dis H^2_{a}(\Omega)=\{u\in H^1_{a}(\Omega): a(x)u_x\in H^1(\Omega)\},
 \end{equation*}
endowed with the norm
 \begin{equation*}\label{}
  \|u\|^2_{H^2_{a}(\Omega)}=\|u\|^2_{H^1_{a}(\Omega)}+\|(a(x)u_x)_x\|^2_{L^2(\Omega)}.
 \end{equation*}

Next, we set
$$
\mathbb{H}:=L^2((0,T);H^1_{a}(\Omega))\cap\mathcal{C}([0,T];L^2(\Omega)).
$$
Let $(H^{1}_a(\Omega))'$ be the topological dual space of $H^1_a(\Omega)$.
If we set
\begin{equation}\label{defWTA}
	W_a(0,T)= \left\{\rho \in L^2((0,T);H^1_a(\Omega)); \rho_t\in L^2\left((0,T);(H^{1}_a(\Omega))^\prime\right)\right\}.
\end{equation}
Then $W_a(0,T)$ endowed with the norm 
\begin{equation}\label{}
	\|\rho\|^2_{	W_a(0,T)}=\|\rho\|^2_{L^2((0,T);H^1_a(\Omega))}+\|\rho_t\|^2_{L^2\left((0,T);(H^{1}_a(\Omega))^\prime\right)}
\end{equation}
is a Hilbert space. Moreover, we have the continuous embedding
\begin{equation}\label{contWTA}
	W_a(0,T)\subset C([0,T],L^2(\Omega)). 
\end{equation}

\begin{remark}\label{remwta}$ $
	
	Under the hypothesis \eqref{k},
$H^1_a(\Omega)$ is compactly embedded in $L^2(\Omega)$ (see \cite{alabau2006}).
\end{remark}
Let us recall the following existence result retrieved from \cite[Page 37]{lions2013}.
\begin{theorem}\label{Theolions61} 
	Let $\left(F, \|\cdot\|_F\right)$ be a Hilbert space. Let $\Phi$ be a subspace of $F$ endowed with a pre-Hilbert scalar product $(((\cdot,\cdot)))$ and the corresponding norm $|||\cdot|||$.  Moreover, let $E:F\times \Phi\to \mathbb{C}$ be a sesquilinear form.  Assume that the following hypothesis hold:
	\begin{enumerate}
		\item  The embedding $\Phi \hookrightarrow F$ is continuous, that is, there is a constant $C_1>0$ such that
		\begin{equation}\label{theolions1}
			\|\varphi\|_{F}\leq C_1|||\varphi|||~~\forall\; \varphi~ in~ \Phi.
		\end{equation}
		
		\item For all $\varphi\in \Phi$, the mapping $u\mapsto E(u,\varphi)$ is continuous on $F$.
		
		\item There is a constant $C_2>0$  such that
		\begin{equation}\label{theolions2}
			{E(\varphi,\varphi)}\geq C_2 |||\varphi|||^2~~~\text{for all}~~\varphi\in \Phi.
		\end{equation}
	\end{enumerate}
	If $\varphi\mapsto L(\varphi)$ is a semi linear continuous form  on $\Phi$, then there exists a function $u\in F$ verifying
	$$
	E(u,\varphi)=L(\varphi)~~ \text{for all} ~~\varphi\in \Phi.
	$$
\end{theorem}

The weak solution of problem \eqref{ywell} is defined as follows.

\begin{dfnt}\label{weaksolution}
	We shall say that a function $y\in \mathbb{H}$ is a weak solution to \eqref{ywell} if the following equality
	\begin{equation}\label{Defweaksolution}
		\begin{array}{lll}
			\dis -\int_Q y\phi_t \, \dq + \int_Q a(x)y_x\phi_x\,\dq+\int_{Q}a_0 y\phi\, \dq+\int_{Q}\beta(x)b_0 y_x\phi\, \dq\\
			=\dis \int_{Q} (h\chi_{\omega}+v^1\chi_{\omega_1}+v^2\chi_{\omega_2})\, \phi\; \dq
			+\int_{\Omega} y^0(x)\,\phi(0,x)\ dx,
		\end{array}
	\end{equation}
	holds, for every 
	\begin{equation}\label{HQ}
		\phi \in \mathbb{V}=\left\{\phi\in \mathbb{H}:\, \phi_t\in L^2(Q), \, \phi(T,\cdot)=0 \hbox{ in } \Omega \right\}.
	\end{equation}
\end{dfnt}

\begin{remark}\label{rmktrace}
	We observe the following:
	\begin{enumerate}[(a)]
		\item The space $\mathbb{V}$ endowed  with the norm defined by
		\begin{align*}
			\|\phi\|^2_{\mathbb{V}}:=\|\phi\|^2_{L^2((0,T);H^1_a(\Omega))}+\|\phi(0,\cdot)\|^2_{L^2(\Omega)}
		\end{align*}
		is a Hilbert space.
		\item If $ \varphi \in \mathbb{V}$, then $\phi_t \in L^2(Q)\hookrightarrow L^2((0,T);(H^{1}_a(\Omega))^\prime)$. Consequently, $\phi \in W_a(0,T)$. Therefore,  $\phi(0,\cdot)$ and $\phi(T,\cdot)$ exist and belong to $L^2(\Omega)$.
	\end{enumerate}
\end{remark}

Using Theorem \ref{Theolions61}, the following proposition shows that the problem \eqref{ywell} is well posed.
 
\begin{prop}\label{exis}$ $
	
	Assume that $a(\cdot)$ satisfies \eqref{k} and Assumptions \ref{assum} are valid. Let $y^0\in L^2(\Omega)$, $a_0,b_0\in L^\infty(Q)$, $h\in L^2(\omega_T)$ and   $(v^1, v^2)^{t}\in \mathcal{H}=L^2((0,T);L^2(\omega_1))\times L^2((0,T);L^2(\omega_2))$. Then, the system \eqref{ywell} admits a unique weak solution $y\in\mathbb{H}$ in the sense of Definition \ref{weaksolution}.
	Moreover, there exists a constant $C=C(T,\|a_0\|_{\infty}, \|b_0\|_{\infty})>0$ such that the following estimation holds:
	\begin{equation}\label{esty1y2}
		\begin{array}{llllll}
			\dis \|y(T,\cdot)\|^2_{L^2(\Omega)}+\|y\|^2_{L^2((0,T); H^1_a(\Omega))}
			\leq 
			C\left(\|v^1\|^2_{L^2(\omega_{1,T})}+\|v^2\|^2_{L^2(\omega_{2,T})}+\|h\|^2_{L^2(\omega_T)}+\|y^0\|^2_{L^2(\Omega)}\right).
		\end{array}
	\end{equation} 	
\end{prop}

The proof of Proposition \ref{exis} can be found in the Appendix.

Similar to the linear problem \eqref{ywell}, we give the definition of weak solution to the semilinear system \eqref{eq}.

\begin{dfnt}\label{weaksolutionsemi}
	A function $y$ is called the weak solution of the problem \eqref{eq} if $y\in \mathbb{H}$ and for any function $\phi\in \mathbb{V}$, the following equality holds
	\begin{equation}\label{Defweaksolutionsemi}
		\begin{array}{lll}
			\dis -\int_Q y\phi_t \, \dq + \int_Q a(x)y_x\phi_x\,\dq+\int_{Q}F(y,y_x) \phi\, \dq\\
			=\dis \int_{Q} (h\chi_{\omega}+v^1\chi_{\omega_1}+v^2\chi_{\omega_2})\, \phi\; \dq
			+\int_{\Omega} y^0(x)\,\phi(0,x)\ dx.
		\end{array}
	\end{equation}
\end{dfnt}
The following result of existence and uniqueness of solutions to system \eqref{eq} is proved in the Appendix.
\begin{theorem}\label{exissemi}$ $

For any $y^0\in L^2(\Omega)$, $h\in L^2(\omega_T)$ and   $(v^1, v^2)^{t}\in \mathcal{H}$, the system \eqref{eq} has a unique weak solution $y\in\mathbb{H}$ in the sense of Definition \ref{weaksolutionsemi}. .
\end{theorem}

We state the Hardy-Poincar\'e inequality which is a standard tool in the analysis of degenerate systems. The following result is proved in \cite[Propsition 2.1]{alabau2006}.

\begin{prop}(Hardy-Poincar\'e inequality) $ $
	
	Assume that $a:[0;1]\longrightarrow\R_+$ is in $\mathcal{C}([0;1])$, $a(0)=0$ and $a>0$ on $(0;1]$. Furthermore, assume that $a$ is such that there exists $\theta\in (0;1)$ such that the function $\dis x\longmapsto\frac{a(x)}{x^\theta}$ is non-increasing in neighbourhood of $x=0$. Then, there is a constant $\overline{C}>0$ such that for any $z$, locally absolutely continuous on $(0;1]$, continuous at $0$ and satisfying $z(0)=0$ and $\dis \int_0^1 a(x)|z^\prime(x)|^2\ dx<+\infty$, the following inequality holds
	
	\begin{equation}\label{hardy}
	\int_0^1 \frac{a(x)}{x^2}|z(x)|^2\ dx<\overline{C}\int_0^1 a(x)|z^\prime(x)|^2\ dx.	
	\end{equation}
	Moreover, under the same hypothesis on $z$ and the fact that the function $\dis x\longmapsto\frac{a(x)}{x^\theta}$ is non-increasing on $(0;1]$, then the inequality \eqref{hardy} holds with $\dis \overline{C}=\frac{4}{(1-\theta)^2}.$	
\end{prop}

\section{Characterization of Nash quasi-equilibrium}\label{low}

In this section, we give the following characterization of Nash quasi-equilibrium pair (see Definition \ref{dfntquasi}) $(\hat{v}^1,\ \hat{v}^2)$ of system \eqref{eq} for functionals $J_i,\ i=1,2$ given by \eqref{all16}.

\begin{prop}\label{quasi}$ $
	
Let $h\in L^2(\omega_T)$ and assume that $\mu_i,\ i=1,2$ are sufficiently large. Let also $(\hat{v}^1,\ \hat{v}^2)\in \mathcal{H}$ be the Nash quasi-equilibrium pair for $(J_1,J_2)$. Then, there exist $p^i\in \mathbb{H}$ such that the Nash quasi-equilibrium pair $(\hat{v}^1, \hat{v}^2)\in \mathcal{H}$ is characterized by
\begin{equation}\label{vop}
	\hat{v}^i=-\frac{1}{\mu_i}p^i\ \ \mbox{in}\ \ \ (0,T)\times \omega_i,\ i=1,2,	
\end{equation}
where  $(y,p^1,p^2)$ is the solution of the following optimality systems
\begin{equation}\label{yop}
	\left\{
	\begin{array}{rllll}
		\dis y_{t}-\left(a(x)y_{x}\right)_{x}+F(y,y_x)  &=&\dis h\chi_{\omega}-\frac{1}{\mu_1}p^1\chi_{\omega_1}-\frac{1}{\mu_2}p^2\chi_{\omega_2}& \mbox{in}& Q,\\
		\dis y(t,0)=y(t,1)&=&0& \mbox{on}& (0,T), \\
		\dis y(0,\cdot)&=&y^0&\mbox{in}&\Omega
	\end{array}
	\right.
\end{equation}
and 
\begin{equation}\label{pop}
	\left\{
	\begin{array}{rllll}
		\dis -p_{t}^i-\left(a(x)p^i_{x}\right)_{x}+D_1F(y,y_x)p^i-\left(\beta(x)\left[\frac{D_2F(y,y_x)}{\beta(x)}\right]p^i\right)_x   &=&\alpha_i\left(y-y_{i,d}\right)\chi_{\omega_{i,d}}& \mbox{in}& Q,\\
		\dis p^i(t,0)=p^i(t,1)&=&0& \mbox{on}& (0,T), \\
		\dis p^i(T,\cdot)&=&0 &\mbox{in}&\Omega.
	\end{array}
	\right.
\end{equation}
\end{prop}

\begin{proof}
	
If $(\hat{v}^1,\ \hat{v}^2)$ is a Nash quasi-equilibrium in the sense of the Definition \ref{dfntquasi}, then we have
\begin{equation}\label{EL}
	\begin{array}{rll}
		&&\dis 	\alpha_i\int_0^T\int_{\omega_{i,d}}(y-y_{i,d})z^i\dT
		 +\mu_i\int_0^T\int_{\omega_i}\hat{v}^iv^i\ \dT=0,\ \mbox{for all}\  v^i\in L^2((0,T);L^2(\omega_i)),
	\end{array}
\end{equation}
where $z^i$ is the solution of 
\begin{equation}\label{zop}
	\left\{
	\begin{array}{rllll}
		\dis z_{t}^i-\left(a(x)z^i_{x}\right)_{x}+D_1F(y,y_x)z^i+\beta(x)\left[\frac{D_2F(y,y_x)}{\beta(x)}\right]z_x^i   &=&v^i\chi_{\omega_{i}}& \mbox{in}& Q,\\
		\dis z^i(t,0)=z^i(t,1)&=&0& \mbox{on}& (0,T), \\
		\dis z^i(0,\cdot)&=&0 &\mbox{in}&\Omega.
	\end{array}
	\right.
\end{equation}

Now, let $p^i$ be the solution of the following (adjoint) system
\begin{equation*}
	\left\{
	\begin{array}{rllll}
		\dis -p_{t}^i-\left(a(x)p^i_{x}\right)_{x}+D_1F(y,y_x)p^i-\left(\beta(x)\left[\frac{D_2F(y,y_x)}{\beta(x)}\right]p^i\right)_x   &=&\alpha_i\left(y-y_{i,d}\right)\chi_{\omega_{i,d}}& \mbox{in}& Q,\\
		\dis p^i(t,0)=p^i(t,1)&=&0& \mbox{on}& (0,T), \\
		\dis p^i(T,\cdot)&=&0 &\mbox{in}&\Omega.
	\end{array}
	\right.
\end{equation*}
If we multiply the first equation of this latter system by $z^i$ solution of \eqref{zop} and we integrate by parts over $Q$, we obtain
\begin{equation*}
	\begin{array}{rll}
		\dis 	\alpha_i\int_0^T\int_{\omega_{i,d}}(y-y_{i,d})z^i\ \dT =\int_0^T\int_{\omega_i}\hat{v}^ip^i\ \dT.
	\end{array}
\end{equation*}
Combining this latter equality with \eqref{EL}, we obtain
\begin{equation*}
	\begin{array}{rll}
		\dis \int_0^T\int_{\omega_i}v^i\left(p^i+\mu_i\hat{v}^i\right)\ \dT=0,\ \mbox{for all}\  v^i\in L^2((0,T);L^2(\omega_i)),
	\end{array}
\end{equation*}
from which we deduce \eqref{vop}-\eqref{pop}.
\end{proof}

Now, let us consider the following linearized systems for \eqref{yop}-\eqref{pop}
\begin{equation}\label{yli}
	\left\{
	\begin{array}{rllll}
		\dis y_{t}-\left(a(x)y_{x}\right)_{x}+a_1y+\beta(x)a_2y_x  &=&\dis h\chi_{\omega}-\frac{1}{\mu_1}p^1\chi_{\omega_1}-\frac{1}{\mu_2}p^2\chi_{\omega_2}& \mbox{in}& Q,\\
		\dis y(t,0)=y(t,1)&=&0& \mbox{on}& (0,T), \\
		\dis y(0,\cdot)&=&y^0&\mbox{in}&\Omega
	\end{array}
	\right.
\end{equation}
and 
\begin{equation}\label{pli}
	\left\{
	\begin{array}{rllll}
		\dis -p_{t}^i-\left(a(x)p^i_{x}\right)_{x}+b_1p^i-(\beta(x)b_2p^i)_x   &=&\alpha_i\left(y-y_{i,d}\right)\chi_{\omega_{i,d}}& \mbox{in}& Q,\\
		\dis p^i(t,0)=p^i(t,1)&=&0& \mbox{on}& (0,T), \\
		\dis p^i(T,\cdot)&=&0 &\mbox{in}&\Omega,
	\end{array}
	\right.
\end{equation}
where $a_i, b_i\in L^\infty(Q),\ i=1,2$.

Next, we consider the adjoint systems for \eqref{yli}-\eqref{pli}.
\begin{equation}\label{rho}
	\left\{
	\begin{array}{rllll}
		\dis -\rho_{t}-\left(a(x)\rho_{x}\right)_{x}+a_1\rho-(\beta(x)a_2\rho)_x &=&\dis \alpha_1\psi^1\chi_{\omega_{1,d}}+\alpha_2\psi^2\chi_{\omega_{2,d}}& \mbox{in}& Q,\\
		\dis \rho(t,0)=\rho(t,1)&=&0& \mbox{on}& (0,T), \\
		\dis \rho(T,\cdot)&=&\rho^T &\mbox{in}&\Omega
	\end{array}
	\right.
\end{equation}
and 
\begin{equation}\label{psi}
	\left\{
	\begin{array}{rllll}
		\dis \psi_{t}^i-\left(a(x)\psi^i_{x}\right)_{x}+b_1\psi^i+\beta(x)b_2\psi_x^i   &=&\dis-\frac{1}{\mu_i}\rho\chi_{\omega_{i}}& \mbox{in}& Q,\\
		\dis\psi^i(t,0)=\psi^i(t,1)&=&0& \mbox{on}& (0,T), \\
		\dis \psi^i(0,\cdot)&=&0 &\mbox{in}&\Omega,
	\end{array}
	\right.
\end{equation}
where $\rho^T\in L^2(\Omega)$.

 \begin{remark}
 	$ $
 	\begin{enumerate}
 		\item Notice that the existence and uniqueness of a solution for \eqref{yop}-\eqref{pop} guaranteed by Proposition \ref{exis} implies the existence and uniqueness of a Nash quasi-equilibrium in the sense of Definition \ref{dfntquasi}.
 		
 		\item Using some ideas of \cite{araruna2015anash, djomegne2021}, we can prove the following estimate
 		\begin{equation}\label{v}
 		\|(\hat{v}_1,\hat{v}_2)\|_{\mathcal{H}}\leq 
 		C\left(\sum_{i=1}^{2}\|y_{i,d}\|_{L^2((0,T)\times\omega_{i,d})}+
 		\|h\|_{L^2(\omega_T)}+\|y^0\|_{L^2(\Omega)}\right),
 		\end{equation} 	
 	where $C=C(\|a_1\|_{\infty},\|a_2\|_{\infty}, T,\alpha_1,\alpha_2)$ is a positive constant.
 	\end{enumerate}	
 \end{remark}

\section{On the convexity of $J_i$, $i=1,2$}\label{conv}

In this section, we prove Proposition \ref{nash} establishing the equivalence between Nash quasi-equilibrium and Nash equilibrium in the semilinear case.
We recall that 
\begin{equation*}
J_i(h;v^1,v^2)=\frac{\alpha_i}{2}\int_0^T\int_{\omega_{i,d}}|y-y_{i,d}|^2\ dxdt+\frac{\mu_i}{2}\int_0^T\int_{\omega_i}|v^i|^2\ dxdt,\ i=1,2.
\end{equation*}
\textbf{Proof of Proposition \ref{nash}.} 

 Let $h\in L^2(\omega_T)$ be given and let $(\hat{v}_1, \hat{v}_2)$ be the associated Nash quasi-equilibrium. For any $w^1,w^2\in L^2(\omega_{1,T})$ and $s\in \R$, we denote by $y^s=(y_1^s,y_2^s)$ the solution of the following system
 	\begin{equation}\label{ys}
 	\left\{
 	\begin{array}{rllll}
 	\dis y^s_{t}-\left(a(x)y^s_{x}\right)_{x}+F(y^s,y_x^s)  &=&h\chi_{\omega}+(v^1+sw^1)\chi_{\omega_1}+v^2\chi_{\omega_2}& \mbox{in}& Q,\\
 	\dis y^s(t,0)=y^s(t,1)&=&0& \mbox{on}& (0,T), \\
 	\dis  y^s(0,\cdot)&=&y^0 &\mbox{in}&\Omega,
 	\end{array}
 	\right.
 	\end{equation}
 	and we set $y:=y^s|_{s=0}$.
 	
 	Then, we have
 	\begin{equation}\label{dj}
 	\begin{array}{llll}
 	&&\dis D_1J_1(h;\hat{v}^1+sw^1,\hat{v}^2)\cdot w^2-	D_1J_1(h;\hat{v}^1,\hat{v}^2)\cdot w^2=s\mu_1\int_0^T\int_{\omega_{1}}w^1w^2\ dxdt\\
 	&&\dis +\alpha_1\int_0^T\int_{\omega_{1,d}}(y^s-y_{1,d})z^{s}\ dxdt-\alpha_1\int_0^T\int_{\omega_{1,d}}(y-y_{1,d})z\ dxdt,\\
 	\end{array}
 	\end{equation}
 	where $z^{s}$ is the derivative of the state $y^s$ with respect to $\hat{v}^1+sw^1$ in the direction $w^2$, that is $z^{s}$ is the solution to
 	\begin{equation}\label{zs}
 	\left\{
 	\begin{array}{rllll}
 	\dis z_{t}^{s}-\left(a(x)z^{s}_{x}\right)_{x}+D_1F(y^s,y_x^s)z^{s}+\beta(x)\left[\frac{D_2F(y^s,y_x^s)}{\beta(x)}\right]z_x^{s}   &=&w^2\chi_{\omega_{1}}& \mbox{in}& Q,\\
 	\dis z^{s}(t,0)=z^{s}(t,1)&=&0& \mbox{on}& (0,T), \\
 	\dis z^{s}(0,\cdot)&=&0 &\mbox{in}&\Omega.
 	\end{array}
 	\right.
 	\end{equation}
 	We will also use the notation $z:=z^{s}|_{s=0}$.
 	
 	Let us introduce the adjoint of \eqref{zs}
 	\begin{equation}\label{ps}
 	\left\{
 	\begin{array}{rllll}
 	\dis -p_{t}^{s}-\left(a(x)p^{s}_{x}\right)_{x}+D_1F(y^s,y_x^s)p^{s}-\left(\beta(x)\left[\frac{D_2F(y^s,y_x^s)}{\beta(x)}\right]p^{s}\right)_x   &=&\alpha_1\left(y^s-y_{1,d}\right)\chi_{\omega_{1,d}}& \mbox{in}& Q,\\
 	\dis p^{s}(t,0)=p^{s}(t,1)&=&0& \mbox{on}& (0,T), \\
 	\dis p^{s}(T,\cdot)&=&0 &\mbox{in}&\Omega
 	\end{array}
 	\right.
 	\end{equation}
 	and let us use the notation $p:=p^{s}|_{s=0}$.\\
 	Multiplying the first equation of \eqref{zs} by $p^{s}$ and integrating by parts over $Q$, we obtain
 	\begin{equation}\label{zp}
 	\alpha_1\int_0^T\int_{\omega_{1,d}}(y^s-y_{1,d})z^{s}\ dxdt=\int_{Q}w^2p^{s}\chi_{\omega_1}\ dxdt.	
 	\end{equation}
 	From \eqref{dj} and \eqref{zp}, we have
 	\begin{equation}\label{dj1}
 	\begin{array}{llll}
 	\dis D_1J_1(h;\hat{v}^1+sw^1,\hat{v}^2)\cdot w^2-	D_1J_1(h;\hat{v}^1,\hat{v}^2)\cdot w^2&=&\dis s\mu_1\int_0^T\int_{\omega_{1}}w^1w^2\ dxdt\\
 	\dis &+&\dis \int_0^T\int_{\omega_{1}}(p^{s}-p)w^2\ dxdt.
 	\end{array}
 	\end{equation}
 	Note that, the following limits exist
 	$$
 	\eta=\lim\limits_{s\to 0}\frac{1}{s}(p^{s}-p)\ \ \mbox{and}\ \ \phi=\lim\limits_{s\to 0}\frac{1}{s}(y^s-y)
 	$$
 	and verify the following systems:
 	\begin{equation}\label{ha}
 		\left\{
 		\begin{array}{rllll}
 			\dis \phi_{t}-\left(a(x)\phi_{x}\right)_{x}+D_1F(y,y_x)\phi+\beta(x)\left[\frac{D_2F(y,y_x)}{\beta(x)}\right]\phi_{x} &=&w^1\chi_{\omega_1}& \mbox{in}& Q,\\
 			\dis \phi(t,0)=\phi(t,1)&=&0& \mbox{on}& (0,T), \\
 			\dis \phi(0,\cdot)&=&0 &\mbox{in}&\Omega
 		\end{array}
 		\right.
 	\end{equation}
 and
 	\begin{equation}\label{eta}
 	\left\{
 	\begin{array}{rllll}
 	\dis -\eta_{t}-\left(a(x)\eta_{x}\right)_{x}+D_1F(y,y_x)\eta-\left(\beta(x)\left[\frac{D_2F(y,y_x)}{\beta(x)}\right]\eta\right)_x+\mathcal{G}(p,\phi)   &=&\alpha_1\phi\chi_{\omega_{1,d}}& \mbox{in}& Q,\\
 	\dis \eta(t,0)=\eta(t,1)&=&0& \mbox{on}& (0,T), \\
 	\dis \eta(T,\cdot)&=&0 &\mbox{in}&\Omega,
 	\end{array}
 	\right.
 	\end{equation}
 	where
 	$$
 	\mathcal{G}(p,\phi)=D_{11}^2F(y,y_x)p\phi+D_{12}^2F(y,y_x)p\phi_x-(D_{21}^2F(y,y_x)p\phi)_x-(D_{22}^2F(y,y_x)p\phi_x)_x.
 	$$
 	
 	Thus, from \eqref{dj1}-\eqref{eta} for $w^2=w^1$, we have 
 	\begin{equation}\label{dj3}
 	\dis D_1^2J_1(h;\hat{v}^1,\hat{v}^2)\cdot(w^1,w^1)= \int_0^T\int_{\omega_{1}}\eta w^1\ dxdt+\mu
 	_1\int_0^T\int_{\omega_{1}}|w^1|^2\ dxdt.
 	\end{equation}
 	
 	From a similar argument as in  the proof of Proposition $1.4$ \cite{araruna2015anash}, we can prove that there exists a constant $C>0$ independent of $h,\ \eta,\ \phi,\ w^1$ such that
 	\begin{equation}\label{a}
 	\left|\int_0^T\int_{\omega_{1}}\eta w^1\ dxdt\right|\leq C\|w^1\|_{L^2(\omega_{1,T})}.
 	\end{equation}
 	Using \eqref{a} together with \eqref{dj3}, we have 
 	\begin{equation*}
 	\dis D_1^2J_1(h;\hat{v}^1,\hat{v}^2)\cdot(w^1,w^1)\geq\left(\mu_1-C\right) \int_{\omega_{1,T}}|w^1|^2\ dxdt,\ \forall w^1\in L^2(\omega_{1,T}).
 	\end{equation*}
 	In a similar way, we show that
 	\begin{equation*}
 	\dis D_2^2J_2(h;\hat{v}^1,\hat{v}^2)\cdot(w^2,w^2)\geq\left(\mu_2-C\right) \int_{\omega_{2,T}}|w^2|^2\ dxdt,\ \forall w^2\in L^2(\omega_{2,T}),
 	\end{equation*}
 for another positive constant $C$ independent of $\mu_1$ and $\mu_2$.
 For $\mu_i$ sufficiently large, the functional $J_i,\ i=1,2$ given by \eqref{all16} are convex and thus, the pair $(\hat{v}^1,\hat{v}^2)$ is a Nash equilibrium in the sense of \eqref{Nash1}. \hfill $\blacksquare$
 
\section{Carleman estimates}\label{Carleman}

The goal of this section is to establish some Carleman estimates for systems \eqref{rho}-\eqref{psi}. Let us begin by introducing some weight functions.

Since $\omega_d\cap \omega\neq \emptyset$ (recall \eqref{oda}), then there exist a non-empty open set $\O_1\Subset \omega_d\cap \omega$ and a function $\sigma\in \mathcal{C}^2([0,1])$ such that
\begin{equation}\label{sigma}
\left\{
\begin{array}{llll}
 \sigma(x)>0\quad\text{in}\quad (0,1),\quad \sigma(0)=\sigma(1)=0,\\
\sigma_x(x)\neq 0\quad \text{in}\quad [0,1]\setminus\O_0,
\end{array}
\right.
\end{equation}
where $\O_0\Subset\O_1$ is an open subset. The existence of such a function $\sigma$ is proved in \cite{FursikovImanuvilov}. \\
Let  $\tau\in [0,1)$ be as in the assumption \eqref{k} and  $r, d\in \R$ be such that 
\begin{equation}\label{condrd}
r\geq \frac{4ln(2)}{\|\sigma\|_{\infty}} \hbox{ and } d\geq \frac{5}{a(1)(2-\tau)}.
\end{equation}
Then, the interval $\dis I= \left[\frac{a(1)(2-\tau)(e^{2r\|\sigma\|_{\infty}}-1)}{d\ a(1)(2-\tau)-1},
\frac{4(e^{2r\|\sigma\|_{\infty}}-e^{r\|\sigma\|_{\infty}})}{3d}\right]$
is non-empty (see  \cite{birba2016}). Let $\lambda\in I$. For $r,\ d$ satisfying \eqref{condrd}, we define the following functions:
\begin{equation}\label{functcarl}
\left\{
\begin{array}{llll}
\dis \Theta(t)=\frac{1}{(t(T-t))^4},\quad \forall t\in (0,T),\ \ \ \delta(x):=\dis \lambda\left(\int_{0}^{x}\frac{y}{a(y)}\ dy-d\right),\\
\\
\dis\varphi(t,x):=\Theta(t)\delta(x),\quad \eta(t,x):=\Theta(t)e^{r\sigma(x)},\\
\\
\Psi(x)=\left(e^{r\sigma(x)}-e^{2r\|\sigma\|_\infty}\right),\ \ \ \Phi(t,x):=\Theta(t)\Psi(x).\\
\end{array}
\right.
\end{equation}
Using the second assumption in \eqref{condrd} on $d$, we observe that
$$
\delta(x)<0,\ \ \forall x\in [0,1].
$$
Moreover, we have that 
$\Theta(t)\to +\infty$ as $t$ tends to $0^+$ and $T^-$.  Under assumptions \eqref{condrd} and the choice of the parameter $\lambda$, the weight functions $\varphi$ and $\Phi$ defined by \eqref{functcarl} satisfy the following inequalities which are needed in the sequel:
	
	\begin{equation}\label{ineqphi}
		\begin{array}{llll}
			\dis \frac{4}{3}\Phi\leq \varphi\leq\Phi\ \ \mbox{on}\ Q,\\
			\dis 2\Phi\leq \varphi\ \ \mbox{on}\ Q.
		\end{array}
	\end{equation}	

To prove the forthcoming theorems, we use the following Carleman estimate in the degenerate case proved in \cite{carmelo2020}.

\begin{prop}\label{prp1}$ $
	
	Consider the following system with $G_0,G_1\in L^2(Q)$ and $z_T\in L^2(\Omega)$,
	
	\begin{equation}\label{defz}
		\left\{
		\begin{array}{rllll}
			\dis -z_t-(a(x)z_x)_x&=&G_0+(\beta(x)G_1)_x &\mbox{in}&Q,\\
			z(t,0)=z(t,1)&=&0  &\mbox{on}& (0,T),\\
			z(T,\cdot) &=&z_T  &\mbox{in}&  \Omega.
		\end{array}
		\right.
	\end{equation}
	Assume that \eqref{k} and \eqref{b} are satisfied and let $T>0$ be given. Then, there exist two positive constants $C$ and $s_0$, such that every solution of \eqref{defz} satisfies, for all $s\geq s_0$, the following inequality:
	\begin{eqnarray}\label{ineqcarl}
		\dis\mathcal{I}(z)\leq C\left(\int_{0}^{T}\int_{\O_1}|z|^2e^{2s\varphi}\,\dq
		\dis +\int_{Q}\left(|G_0|^2+s^2\Theta^3\frac{\beta^2(x)}{a(x)}|G_1|^2\right)e^{2s\varphi}\,\dq\right),
	\end{eqnarray}	
	where 
	\begin{equation}\label{I}
		\mathcal{I}(z)=\int_{Q}\left(s^3\Theta^3\frac{x^2}{a(x)}z^2+s\Theta a(x)z_x^2\right)e^{2s\varphi}\, \dq	
	\end{equation}
	and the functions $\Theta$ and $\varphi$ are given by \eqref{functcarl}.

\end{prop}

Using this latter proposition, we prove the following result.

\begin{prop}\label{prp2}$ $
	
Let us consider
\begin{equation}\label{defz1}
	\left\{
	\begin{array}{rllll}
		\dis -z_t-(a(x)z_x)_x+a_0z&=&G_2+(\beta(x)G_1)_x &\mbox{in}&Q,\\
		z(t,0)=z(t,1)&=&0  &\mbox{on}& (0,T),\\
		z(T,\cdot) &=&z_T  &\mbox{in}&  \Omega,
	\end{array}
	\right.
\end{equation}	
where $G_1,G_2\in L^2(Q)$, $a_0\in L^\infty(Q)$ and $z_T\in L^2(\Omega)$.	

	Assume that \eqref{k} and \eqref{b} are satisfied and let $T>0$ be given. Then, there exist two positive constants $C$ and $s_1$, such that every solution of \eqref{defz1} satisfies, for all $s\geq s_1$, the following inequality:
	
\begin{eqnarray}\label{ineqcarl1}
	\dis\mathcal{I}(z)\leq C\left(\int_{0}^{T}\int_{\O_1}|z|^2e^{2s\varphi}\,\dq
	\dis +\int_{Q}\left(|G_2|^2+s^2\Theta^3\frac{\beta^2(x)}{a(x)}|G_1|^2\right)e^{2s\varphi}\,\dq\right),
\end{eqnarray}	
where $\mathcal{I}(\cdot)$ is given by \eqref{I}.
	
\end{prop}

\begin{proof}
	To prove the inequality \eqref{ineqcarl1}, we apply the relation \eqref{ineqcarl} with $G_0=G_2-a_0z$.
	Hence, there are two positive constants $C$ and $s_0$, such that for all $s\geq s_0$, the following inequality holds:
	\begin{eqnarray*}
		\dis\dis\mathcal{I}(z)\leq C\left(\int_{0}^{T}\int_{\O_1}|z|^2e^{2s\varphi}\,\dq
		\dis +\int_{Q}\left(|G_0|^2+s^2\Theta^3\frac{\beta^2(x)}{a(x)}|G_1|^2\right)e^{2s\varphi}\,\dq\right).
	\end{eqnarray*}
	On the other hand, using Young's inequality, we have 
	$$
	\int_{Q}|G_0|^2e^{2s\varphi}\,\dq\leq 2\left(\int_{Q}|G_2|^2e^{2s\varphi}\,\dq+\|a_0\|^2_{\infty}\int_{Q}|z|^2e^{2s\varphi}\,\dq\right).
	$$
	Now, applying Hardy-Poincar\'e inequality \eqref{hardy} to the function $e^{s\varphi}z$, the fact that $\dis x\longmapsto\frac{x^2}{a(x)}$ is non-decreasing and thanks to the definition of $\varphi$, it follows that
	
	\begin{equation*}
		\begin{array}{rll}
			\dis \int_{Q}|z|^2e^{2s\varphi}\,\dq &\leq&\dis  \frac{1}{a(1)}\int_Q\frac{a(x)}{x^2}|z|^2e^{2s\varphi}\,\dq\\
			&\leq&\dis  \frac{C}{a(1)}\int_Qa(x)\left((e^{s\varphi}z)_x\right)^2\,\dq\\
			&\leq&\dis  \frac{C}{a(1)}\left(\int_Qs^2\lambda^2\Theta^2\frac{x^2}{a(x)}e^{2s\varphi}z^2\,\dq+\int_Qa(x)e^{2s\varphi}z_x^2\,\dq\right).
		\end{array}
	\end{equation*}
	Thus,
	\begin{equation*}
		\begin{array}{rll}
			\dis \int_{Q}|G_0|^2e^{2s\varphi}\,\dq&\leq&\dis 2\int_{Q}|G_2|^2e^{2s\varphi}\,\dq\\
			\dis &&\dis +2\|a_0\|^2_{\infty}\frac{C}{a(1)}\left(\int_Qs^2\lambda^2\Theta^2\frac{x^2}{a(x)}e^{2s\varphi}z^2\,\dq+\int_Qa(x)e^{2s\varphi}z_x^2\,\dq\right).
		\end{array}
	\end{equation*}
	Using the fact that there exists a positive constant $M_3$ such that 
	\begin{equation}\label{theta}
		1\leq M_3\Theta\ \ \mbox{and}\ \ \Theta^2\leq M_3\Theta^3,	
	\end{equation}
	we obtain 
	\begin{equation*}
		\begin{array}{rll}
			\dis \mathcal{I}(z)&\leq&
			 \dis C\left(\int_{0}^{T}\int_{\O_1}|z|^2e^{2s\varphi}\,\dq
			 \dis +\int_{Q}\left(|G_2|^2+s^2\Theta^3\frac{\beta^2(x)}{a(x)}|G_1|^2\right)e^{2s\varphi}\,\dq\right)\\
			&&\dis+C_0 \int_{Q}\left(s^2\Theta^3\frac{x^2}{a(x)}z^2+\Theta a(x)z_x^2\right)e^{2s\varphi}\,\dq.
		\end{array}
	\end{equation*}
	Taking $s\geq s_1=\max(s_0, 2C_0)$, we obtain \eqref{ineqcarl1}. This completes the proof.	
\end{proof}

The next result is the classical Carleman estimate in a suitable interval $(b_1,b_2)\subset [0,1]$ proved in \cite{doubova2002, guerrero2006}.

\begin{prop}\label{cobr}$ $
	
	We consider the following system with  $H,H_0\in L^2(Q_b)$ and $a\in \mathcal{C}^1([b_1;b_2])$ is a strictly positive function,
	\begin{equation}\label{coba}
		\left\{
		\begin{array}{rllll}
			\dis -z_t-(a(x)z_x)_x &=&H_0+H_x &\mbox{in}&Q_b:=(0,T)\times (b_1,b_2),\\
			z(t,b_1)=z(t,b_2)&=&0  &\mbox{on}& (0,T).
		\end{array}
		\right.
	\end{equation}	
	Then, there exist two positive constants $C$ and $s_2$, such that every solution of \eqref{coba} satisfies, for all $s\geq s_2$, the following inequality holds
	
	\begin{eqnarray}\label{cobb}
		\mathcal{K}(z) \leq C\left(\int_{0}^{T}\int_{\O_1}s^3\eta^3|z|^2e^{2s\Phi}\,\dq
		\dis +\int_{Q_b}\left(|H_0|^2+s^2\eta^2|H|^2\right)e^{2s\Phi}\,\dq\right),
	\end{eqnarray}
	where
	\begin{equation}\label{dk}
		\mathcal{K}(z)=\int_{Q_b}(s^3\eta^3z^2+s\eta z_x^2)e^{2s\Phi}\,\dq
	\end{equation}
	and the functions $\eta$ and $\Phi$ are defined as in \eqref{functcarl}.
\end{prop}

Now, we state and prove the last result of this section.

\begin{prop}\label{propcarl2}$ $
	
	We consider the following system
	\begin{equation}\label{carl4}
		\left\{
		\begin{array}{rllll}
			\dis -z_t-(a(x)z_x)_x+a_0z &=&H_1+H_x &\mbox{in}&Q_b:=(0,T)\times (b_1,b_2),\\
			z(t,b_1)=z(t,b_2)&=&0  &\mbox{on}& (0,T),\\
		\end{array}
		\right.
	\end{equation}
where $H,H_0\in L^2(Q_b)$, $a_0\in L^\infty(Q)$ and $a\in \mathcal{C}^0([b_1;b_2])$ is a strictly positive function.
	
	Then, there exist two positive constants $C$ and $s_3$, such that every solution of \eqref{coba} satisfies the following inequality holds for all $s\geq s_3$
	
	\begin{eqnarray}\label{ineqcarl2}
		\mathcal{K}(z) \leq C\left(\int_{0}^{T}\int_{\O_1}s^3\eta^3|z|^2e^{2s\Phi}\,\dq
		\dis +\int_{Q_b}\left(|H_1|^2+s^2\eta^2|H|^2\right)e^{2s\Phi}\,\dq\right),
	\end{eqnarray}
	where $ \mathcal{K}(\cdot)$ is defined in \eqref{dk}.
	
\end{prop}

\begin{proof} 
	To prove the relation \eqref{ineqcarl2}, we apply the relation \eqref{cobb} with $H_0=H_1-a_0z$.
	Therefore, there are two positive constants $C$ and $s_2$, such that for all $s\geq s_2$, we have
	\begin{eqnarray*}\label{}
		\mathcal{K}(z) &\leq& C\left(\int_{0}^{T}\int_{\O_1}s^3\eta^3|z|^2e^{2s\Phi}\,\dq
		\dis +\int_{Q_b}\left(|H_0|^2+s^2\eta^2|H|^2\right)e^{2s\Phi}\,\dq\right)\\
		&+&\dis\|a_0\|^2_{\infty}C\int_{Q_b}|z|^2e^{2s\Phi}\,\dq.
	\end{eqnarray*}
	Observing that $s>1$ and $\eta^{-1}\in L^{\infty}(Q_b)$, it follows from the latter inequality that, there exists $C>0$ such that
		\begin{eqnarray*}\label{}
		\mathcal{K}(z) &\leq& C\left(\int_{0}^{T}\int_{\O_1}s^3\eta^3|z|^2e^{2s\Phi}\,\dq
		\dis +\int_{Q_b}\left(|H_1|^2+s^2\eta^2|H|^2\right)e^{2s\Phi}\,\dq\right)\\
		&+&\dis \|a_0\|^2_{\infty}C\int_{Q_b}s^2\eta^3|z|^2e^{2s\Phi}\,\dq.
	\end{eqnarray*}
	Choosing $s\geq s_3=\max \left(s_2, 2\|a_0\|^2_{\infty}C\right)$ in this latter inequality, we obtain \eqref{ineqcarl2}. This completes the proof.
\end{proof}

\subsection{An intermediate Carleman estimate}

In this section, we establish an observability inequality for the adjoint systems \eqref{rho}-\eqref{psi}. This inequality will allows us to prove the null controllability of system \eqref{yop}-\eqref{pop}.  

Since $\dis \omega_{1,d}=\omega_{2,d}:=\omega_d$ (see \eqref{od}) and if we set $\varrho=\alpha_1\psi^1+\alpha_2\psi^2$, then we can simplify \eqref{rho}-\eqref{psi} as follows
\begin{equation}\label{rhobis}
	\left\{
	\begin{array}{rllll}
		\dis -\rho_{t}-\left(a(x)\rho_{x}\right)_{x}+a_1\rho-(\beta(x)a_2\rho)_x  &=&\dis \varrho\chi_{\omega_{d}}& \mbox{in}& Q,\\
		\dis \rho(t,0)=\rho(t,1)&=&0& \mbox{on}& (0,T), \\
		\dis \rho(T,\cdot)&=&\rho^T &\mbox{in}&\Omega
	\end{array}
	\right.
\end{equation}
and 
\begin{equation}\label{varrho}
	\left\{
	\begin{array}{rllll}
		\dis \varrho_{t}-\left(a(x)\varrho_{x}\right)_{x}+b_1\varrho+\beta(x)b_2\varrho_x   &=&\dis-\left(\frac{\alpha_1}{\mu_1}\chi_{\omega_{1}}+\frac{\alpha_2}{\mu_2}\chi_{\omega_{2}}\right)\rho& \mbox{in}& Q,\\
		\dis \varrho(t,0)=\varrho(t,1)&=&0& \mbox{on}& (0,T), \\
		\dis \varrho(0,\cdot)&=&0 &\mbox{in}&\Omega.
	\end{array}
	\right.
\end{equation}

Before going further, we consider the following result useful for the rest of the paper.

\begin{leme}(Caccioppoli's inequality)\cite{maniar2011}\label{cacc}$ $
	
	Let $\O^\prime$ be a subset of $\O_1$ such that $\O^\prime\Subset\O_1$. Let $\rho$ and $\varrho$ be the solution of \eqref{rhobis} and \eqref{varrho} respectively. Then, there exists a positive constant $C$ such that
	\begin{equation}\label{caccio}
		\int_{0}^{T}\int_{\O^\prime}(\rho^2_{x}+\varrho_{x}^2)\ e^{2s\varphi}\,\dq \leq C\int_0^T\int_{\O_1}s^2\Theta^2(\rho^2+\varrho^2)\ e^{2s\varphi}\ \dq,
	\end{equation}	
	where the weight functions $\varphi$ and $\Theta$ are defined by \eqref{functcarl}.
\end{leme}

Now, we state and prove one of the important result of this paper which is the intermediate Carleman inequality for the solutions of systems \eqref{rhobis}-\eqref{varrho}. 

\begin{theorem}\label{thm1}$ $
	
	Assume that the coefficient $a(\cdot)$ verifies \eqref{k}. Then, there exists a  constant $C_1>0$ such that every solution $\rho$ and $\varrho$ of \eqref{rhobis} and \eqref{varrho} respectively, satisfy, for any $s$ large enough, the following inequality
	
	\begin{eqnarray}\label{ineqcarlprinc}
	\mathcal{I}(\rho)+\mathcal{I}(\varrho) \leq C_1\int_{0}^{T}\int_{\O_1}s^3\Theta^3(\rho^2+\varrho^2)e^{2s\Phi}\,\dq,
	\end{eqnarray}
  where the notation $\mathcal{I}(\cdot)$ is defined by \eqref{I}.
\end{theorem}

\begin{proof}
Let us choose an arbitrary open subset $\O^\prime:=(\alpha,\beta)$ such that $\O^\prime\Subset\O_1$ and consider the smooth cut-off function $\xi:\R\to\R$ defined as follows
	\begin{equation}\label{cutoff}
	\left\{
	\begin{array}{llll}
	\dis 0\leq \xi\leq 1, \quad x\in \Omega,\\
	\dis \xi(x)=1, \quad x\in [0,\alpha],\\
	\dis \xi(x)=0, \quad x\in [\beta,1].
	\end{array}
	\right.
	\end{equation}
	Let $\widetilde{\rho}=\xi \rho$ and $\widetilde{\varrho}=\xi \varrho$ where $(\rho, \varrho)$ is the solution of \eqref{rhobis}-\eqref{varrho}. Then, $\widetilde{\rho}$ and $\widetilde{\varrho}$  satisfy the following systems
\begin{equation}\label{rhobist}
	\left\{
	\begin{array}{rllll}
		\dis -\widetilde{\rho}_{t}-\left(a(x)\widetilde{\rho}_{x}\right)_{x}+a_1\widetilde{\rho}-(\beta(x)a_2\widetilde{\rho})_x  &=&\dis \widetilde{\mathcal{G}}_1& \mbox{in}& Q,\\
		\dis \widetilde{\rho}(t,0)=\widetilde{\rho}(t,1)&=&0& \mbox{on}& (0,T), \\
		\dis \widetilde{\rho}(T,\cdot)&=&\widetilde{\rho}_2^T &\mbox{in}&\Omega
	\end{array}
	\right.
\end{equation}
and 
\begin{equation}\label{varrhot}
	\left\{
	\begin{array}{rllll}
		\dis \widetilde{\varrho}_{t}-\left(a(x)\widetilde{\varrho}_{x}\right)_{x}+b_1\widetilde{\varrho}+\beta(x)b_2\widetilde{\varrho}_x   &=&\dis\widetilde{\mathcal{G}}_2& \mbox{in}& Q,\\
		\dis \widetilde{\varrho}(t,0)=\widetilde{\varrho}(t,1)&=&0& \mbox{on}& (0,T), \\
		\dis \widetilde{\varrho}_2(0,\cdot)&=&0 &\mbox{in}&\Omega,
	\end{array}
	\right.
\end{equation}
where 
$$
\widetilde{\mathcal{G}}_1=\widetilde{\varrho}\chi_{\omega_d}-\left(a(x)\xi_x\ \rho\right)_{x}-\xi_x\  a(x)\rho_{x}-\beta(x)a_2\xi_x\rho
$$
and
$$\dis \widetilde{\mathcal{G}}_2=-\left(\frac{\alpha_1}{\mu_1}\chi_{\omega_{1}}+\frac{\alpha_2}{\mu_2}\chi_{\omega_{2}}\right)\widetilde{\rho}-\left(a(x)\xi_x\ \varrho\right)_{x}-\xi_x\  a(x)\varrho_{x}+\beta(x)b_2\xi_x\varrho
.$$

Applying the Carleman estimate \eqref{ineqcarl1} to system \eqref{rhobist} with $G_1=a_2\widetilde{\rho}$ and $G_2=\widetilde{\mathcal{G}}_1$, we obtain
	\begin{eqnarray}\label{ineqcarl3rho}
		\mathcal{I}(\widetilde{\rho})
		&\leq&C \dis \int_{Q}\left[|\widetilde{\varrho}\chi_{\omega_d}|^2+|\left(a(x)\xi_x\ \rho\right)_{x}+\xi_x\  a(x)\rho_{x}|^2+|\beta(x)a_2\xi_x\rho|^2\right] e^{2s\varphi}\,\dq\nonumber\\
		&+&C\int_{0}^{T}\int_{\O_1}e^{2s\varphi}|\widetilde{\rho}|^2\,\dq+C\int_{Q}s^2\Theta^3\frac{\beta^2(x)}{a(x)}|a_2\widetilde{\rho}|^2e^{2s\varphi}\,\dq.
	\end{eqnarray}
	From the definition of $\xi$, we have
	\begin{eqnarray}\label{ine}
	\int_{Q} |(a(x)\xi_x{\rho})_x+a(x)\xi_x{\rho}_{x}|^2 e^{2s\varphi}\,\dq&=& \int_{Q} ((a(x)\xi_x)_x{\rho}+2a(x)\xi_x{\rho}_{x})^2 e^{2s\varphi}\,\dq\nonumber\\
	&\leq& \int_{Q}\left[2((a(x)\xi_x)_x)^2{|\rho|}^2+8(a(x)\xi_x)^2|\rho_{x}|^2\right] e^{2s\varphi}\,\dq\nonumber\\
	&\leq& C\int_0^T\int_{\O^\prime}(|\rho|^2+|\rho_{x}|^2)\ e^{2s\varphi}\,\dq.
	\end{eqnarray}
	In the other hand, using the fact that $\dis x\mapsto \frac{x^2}{a(x)}$ is non-decreasing on $(0,1]$, thanks to Hardy-Poincar\'e inequality \eqref{hardy} applied to the function $e^{s\varphi}\widetilde{\varrho}$ and using the definition of $\varphi$, we get
	\begin{eqnarray*}
		\int_{Q}|\widetilde{\varrho}\chi_{\omega_d}|^2e^{2s\varphi}\, \dq
		&\leq& \frac{1}{a(1)}\int_{Q} \frac{a(x)}{x^2}\widetilde{\varrho}^2e^{2s\varphi}\, \dq\\
		&\leq& \frac{\overline{C}}{a(1)}\int_{Q}a(x)|(\widetilde{\varrho}\ e^{s\varphi})_x|^2\, \dq\\
		&\leq& C\left(\int_{Q}a(x)\widetilde{\varrho}_{x}^2e^{2s\varphi}\, \dq+\int_{Q}s^2\Theta^2\frac{x^2}{a(x)}\widetilde{\varrho}^2e^{2s\varphi}\ \dq\right).
	\end{eqnarray*}
	Using \eqref{theta}, we obtain 
	\begin{eqnarray}\label{ine1}
	\int_{Q}|\widetilde{\varrho}\chi_{\omega_d}|^2e^{2s\varphi}\, \dq\leq  C\left(\int_{Q}\Theta a(x)\widetilde{\varrho}_{x}^2e^{2s\varphi}\, \dq+\int_{Q}s^2\Theta^3\frac{x^2}{a(x)}\widetilde{\varrho}^2e^{2s\varphi}\ \dq\right).
	\end{eqnarray}
Proceeding as for \eqref{ine1}, one obtains 
\begin{eqnarray}
	\int_{0}^{T}\int_{\O_1}e^{2s\varphi}|\widetilde{\rho}|^2\, \dq\leq  C\left(\int_{Q}\Theta a(x)\widetilde{\rho}_{x}^2e^{2s\varphi}\, \dq+\int_{Q}s^2\Theta^3\frac{x^2}{a(x)}\widetilde{\rho}^2e^{2s\varphi}\ \dq\right).
\end{eqnarray}
We have
\begin{eqnarray*}
	\int_{Q}s^2\Theta^3\frac{\beta^2(x)}{a(x)}|a_2\widetilde{\rho}|^2e^{2s\varphi}\,\dq\leq \|a_2\|^2_{\infty}\int_{Q}s^2\Theta^3\frac{\beta^2(x)}{a(x)}|\widetilde{\rho}|^2e^{2s\varphi}\,\dq.
\end{eqnarray*}
Using \eqref{bet}, we have
\begin{equation*}
	\frac{\beta^2(x)}{a(x)}	\leq C\frac{x^2}{a(x)},
\end{equation*}
and therefore
\begin{eqnarray}
	\int_{Q}s^2\Theta^3\frac{\beta^2(x)}{a(x)}|a_2\widetilde{\rho}|^2e^{2s\varphi}\,\dq\leq C(|a_2\|_{\infty})\int_{Q}s^2\Theta^3\frac{x^2}{a(x)}|\widetilde{\rho}|^2e^{2s\varphi}\,\dq.
\end{eqnarray}
Using the fact that $\dis |\beta(x)|\leq L\sqrt{a(x)}$ (see \eqref{b1}) and the definition of $\xi$, we obtain
\begin{eqnarray*}
\int_{Q}|\beta(x)a_2\xi_x\rho|^2 e^{2s\varphi}\,\dq	&\leq&\|a_2\|^2_{\infty}L^2\int_{Q}a(x)\xi_x^2|\rho|^2 e^{2s\varphi}\,\dq\\
&\leq&\|a_2\|^2_{\infty}L^2\int_{0}^{T}\int_{\O^\prime}a(x)
|\rho|^2 e^{2s\varphi}\,\dq.	
\end{eqnarray*}
Due to the fact that $a(x)$ is bounded for all $x\in \O^\prime$, we deduce
\begin{eqnarray}\label{ca}
	\int_{Q}|\beta(x)a_2\xi_x\rho|^2 e^{2s\varphi}\,\dq	
	\leq C(\|a_2\|_{\infty})\int_{0}^{T}\int_{\O^\prime}
	|\rho|^2 e^{2s\varphi}\,\dq.	
\end{eqnarray}

	Combining \eqref{ineqcarl3rho}-\eqref{ca}, we obtain
	\begin{eqnarray}\label{i1}
	\mathcal{I}(\widetilde{\rho})
	\leq  C\int_0^T\int_{\O^\prime}(\rho^2+\rho_{x}^2)e^{2s\varphi}\,\dq 
	+C\int_{Q}\left(\Theta a(x)\widetilde{\varrho}_{x}^2+s^2\Theta^3\frac{x^2}{a(x)}\widetilde{\varrho}^2\right)e^{2s\varphi}\ \dq\nonumber\\
	+C\int_{Q}\left(\Theta a(x)\widetilde{\rho}_{x}^2+s^2\Theta^3\frac{x^2}{a(x)}\widetilde{\rho}^2\right)e^{2s\varphi}\ \dq.
	\end{eqnarray}
	Arguing in the same way as in \eqref{i1} with $\widetilde{\varrho}$ solution of \eqref{varrhot}, we get
	\begin{eqnarray}\label{i4}
		\mathcal{I}(\widetilde{\varrho})
		\leq  C\int_0^T\int_{\O^\prime}(\varrho^2+\varrho_{x}^2)e^{2s\varphi}\,\dq 
		+C\int_{Q}\left(\Theta a(x)\widetilde{\varrho}_{x}^2+s^2\Theta^3\frac{x^2}{a(x)}\widetilde{\varrho}^2\right)e^{2s\varphi}\ \dq\nonumber\\
		+C\int_{Q}\left(\Theta a(x)\widetilde{\rho}_{x}^2+s^2\Theta^3\frac{x^2}{a(x)}\widetilde{\rho}^2\right)e^{2s\varphi}\ \dq.
	\end{eqnarray}
	Adding \eqref{i1} and \eqref{i4} and taking $s$ large enough, we obtain
		
	\begin{eqnarray*}
		\dis \mathcal{I}(\widetilde{\rho})+\mathcal{I}(\widetilde{\varrho}) \leq C\int_0^T\int_{\O^\prime}(\rho^2+\varrho^2+\rho^2_{x}+\varrho^2_{x})e^{2s\varphi}\,\dq.
	\end{eqnarray*}
	Using Caccioppoli's inequality \eqref{caccio}, this latter estimate becomes
	
	\begin{eqnarray}\label{i5}
		\dis \mathcal{I}(\widetilde{\rho})+\mathcal{I}(\widetilde{\varrho}) \leq C\int_0^T\int_{\O_1}s^2\Theta^2(\rho^2+\varrho^2)e^{2s\varphi}\,\dq.
	\end{eqnarray}
	
Now let $\overline{\rho}=\vartheta\rho$ and $\overline{\varrho}= \vartheta\varrho$ where $\vartheta=1-\xi$. Then, the supports of $\overline{\rho}$ and $\overline{\varrho}$ are contained in $[0,T]\times[\alpha,1]$ and satisfy
\begin{equation}\label{rhobisb}
	\left\{
	\begin{array}{rllll}
		\dis -\overline{\rho}_{t}-\left(a(x)\overline{\rho}_{x}\right)_{x}+a_1\overline{\rho}-(\beta(x)a_2\overline{\rho})_x &=&\dis\overline{\mathcal{G}}_1& \mbox{in}& Q_{\alpha}=(0,T)\times(\alpha,1),\\
		\dis \overline{\rho}(t,0)=\overline{\rho}(t,1)&=&0& \mbox{on}& (0,T), \\
		\dis \overline{\rho}_1^T,\ \ \overline{\rho}(T,\cdot)&=&\overline{\rho}^T &\mbox{in}&\Omega,
	\end{array}
	\right.
\end{equation}
and 
\begin{equation}\label{varrhob}
	\left\{
	\begin{array}{rllll}
		\dis \overline{\varrho}_{t}-\left(a(x)\overline{\varrho}_{x}\right)_{x}+b_1\overline{\varrho}+\beta(x)b_2\overline{\varrho}_{x}   &=&\dis\overline{\mathcal{G}}_2& \mbox{in}& Q_{\alpha},\\
		\dis \overline{\varrho}(t,0)=\overline{\varrho}(t,1)&=&0& \mbox{on}& (0,T), \\
		\dis \overline{\varrho}(0,\cdot)&=&0 &\mbox{in}&\Omega,
	\end{array}
	\right.
\end{equation}
where 
$$
\overline{\mathcal{G}}_1=\overline{\varrho}\chi_{\O_d}-\left(a(x)\vartheta_x\ \rho\right)_{x}-\vartheta_x\  a(x)\rho_{x}-\beta(x)a_2\vartheta_x\rho
$$
and
$$\dis \overline{\mathcal{G}}_2=-\left(\frac{\alpha_1}{\mu_1}\chi_{\omega_{1}}+\frac{\alpha_2}{\mu_2}\chi_{\omega_{2}}\right)\overline{\rho}-\left(a(x)\vartheta_x\ \varrho\right)_{x}-\vartheta_x\  a(x)\varrho_{x}+\beta(x)b_2\xi_x\varrho
.$$

Since on $Q_{\alpha}$ all the above systems are non degenerate, applying the classical Carleman estimate \eqref{ineqcarl2}
 to system \eqref{rhobisb} with $H=\beta(x)a_2\overline{\rho}$ and $H_1=\overline{\mathcal{G}}_1$, one obtains
\begin{eqnarray}\label{cl}
	\mathcal{K}(\overline{\rho})
	&\leq& C \dis \int_{Q}\left[|\overline{\varrho}\chi_{\omega_d}|^2+|(a(x)\vartheta_x{\rho})_x+a(x)\vartheta_x{\rho}_{x}|^2+|\beta(x)a_2\vartheta_x\rho|^2\right] e^{2s\Phi}\,\dq\nonumber\\
	&+&C\int_0^T\int_{\O_1}s^3\Theta^3e^{2s\Phi}|\overline{\rho}|^2\,\dq+C\int_{Q}s^2\eta^2e^{2s\Phi}|\beta(x)a_2\overline{\rho}|^2\,\dq,
\end{eqnarray}
because $\dis\eta(t,x):=\Theta(t)e^{r\sigma(x)}\leq \Theta(t)e^{r\|\sigma(x)\|_{\infty}}$.

Using the definition of the function $\vartheta$, we have
\begin{eqnarray}\label{}
	\int_{Q} |(a(x)\vartheta_x{\rho})_x+a(x)\vartheta_x{\rho}_{x}|^2 e^{2s\Phi}\,\dq&=& \int_{Q} ((a(x)\vartheta_x)_x{\rho}+2a(x)\vartheta_x{\rho}_{x})^2 e^{2s\Phi}\,\dq\nonumber\\
	&\leq& \int_{Q}\left[2((a(x)\vartheta_x)_x)^2{\rho}^2+8(a(x)\vartheta_x)^2{\rho}_{x}^2\right] e^{2s\Phi}\,\dq\nonumber\\
	&\leq& C\int_0^T\int_{\O^\prime}(\rho^2+\rho_{x}^2)\ e^{2s\Phi}\,\dq.
\end{eqnarray}
On the other hand, since $\dis x\mapsto \frac{x^2}{a(x)}$ is non-decreasing on $(0,1]$ and thanks to Hardy-Poincar\'e inequality \eqref{hardy} to the function $e^{s\Phi}\overline{\varrho}$, we get
\begin{eqnarray*}
	\int_{Q}|\overline{\varrho}\chi_{\omega_d}|^2e^{2s\Phi}\, \dq
	&\leq& \frac{1}{a(1)}\int_{Q} \frac{a(x)}{x^2}\overline{\varrho}^2e^{2s\Phi}\, \dq\\
	&\leq& \frac{\overline{C}}{a(1)}\int_{Q}a(x)|(\overline{\varrho}\ e^{s\Phi})_x|^2\, \dq\\
&\leq& C\int_{Q}\left(a(x)\overline{\varrho}_{x}^2+a(x)s^2\eta^2\overline{\varrho}^2\right)e^{2s\Phi}\ \dq.
\end{eqnarray*}
We have $\overline{\rho}=\vartheta\rho$ and since $\vartheta$ is bounded on $\O_1$, we deduce that
\begin{eqnarray}
\int_0^T\int_{\O_1}s^3\Theta^3e^{2s\Phi}|\overline{\rho}|^2\,\dq	\leq\int_0^T\int_{\O_1}s^3\Theta^3e^{2s\Phi}|\rho|^2\,\dq.
\end{eqnarray}
Due to \eqref{theta}, the fact that $a\in \mathcal{C}([\alpha;1])$ and $\eta^{-1}\in L^\infty(Q)$, we get
\begin{eqnarray}\label{i8}
\int_{Q}\overline{\varrho}^2e^{2s\Phi}\, \dq\leq  C\int_{Q}\left(\eta \overline{\varrho}_{x}^2+s^2\eta^3\overline{\varrho}^2\right)e^{2s\Phi}\ \dq.
\end{eqnarray}
Using the fact that $\dis |\beta(x)|\leq L\sqrt{a(x)}$ (see Remark \ref{rmq}) and the definition of $\vartheta$, we have
\begin{eqnarray*}
	\int_{Q}|\beta(x)a_2\vartheta_x\rho|^2 e^{2s\Phi}\,\dq	&\leq&\|a_2\|^2_{\infty}L^2\int_{Q}a(x)\vartheta_x^2|\rho|^2 e^{2s\Phi}\,\dq\\
	&\leq&\|a_2\|^2_{\infty}L^2\int_{0}^{T}\int_{\O^\prime}a(x)
	|\rho|^2 e^{2s\Phi}\,\dq.	
\end{eqnarray*}
As the coefficient $a(\cdot)$ is bounded in $\O^\prime$, then we deduce
\begin{eqnarray}\label{}
	\int_{Q}|\beta(x)a_2\vartheta_x\rho|^2 e^{2s\Phi}\,\dq	
	\leq C(\|a_2\|_{\infty})\int_{0}^{T}\int_{\O^\prime}
	|\rho|^2 e^{2s\Phi}\,\dq.	
\end{eqnarray}
Since $\dis |\beta(x)|\leq L\sqrt{a(x)}$, one obtains
\begin{eqnarray*}
	\int_{Q}s^2\eta^2e^{2s\Phi}|\beta(x)a_2\overline{\rho}|^2\,\dq
	\leq\|a_2\|^2_{\infty}L^2\int_{Q}s^2\eta^2a(x)
	e^{2s\Phi}|\overline{\rho}|^2\,\dq.	
\end{eqnarray*}
Using the fact that $a$ is bounded in $(0,1]$ and $\eta^{-1}\in L^\infty(Q)$ we deduce
\begin{eqnarray}\label{i7}
	\int_{Q}s^2\eta^2e^{2s\Phi}|\beta(x)a_2\overline{\rho}|^2\,\dq
	\leq C(\|a_2\|_{\infty})\int_{Q}s^2\eta^3
	e^{2s\Phi}|\overline{\rho}|^2\,\dq	.	
\end{eqnarray}
	Combining \eqref{cl}-\eqref{i7}, we obtain
	
	\begin{eqnarray}\label{cm}
	\mathcal{K}(\overline{\rho})
	\leq  C\int_0^T\int_{\O^\prime}(\rho^2+\rho_{x}^2)e^{2s\Phi}\,\dq+C\int_0^T\int_{\O_1}s^3\Theta^3\rho^2 e^{2s\Phi}\,\dq\nonumber\\
	+C\int_{Q}\left(\eta \overline{\varrho}_{x}^2+s^2\eta^3\overline{\varrho}^2\right)e^{2s\Phi}\ \dq+C\int_{Q}s^2\eta^3\overline{\rho}^2e^{2s\Phi}\ \dq.
	\end{eqnarray}	
Arguing as in \eqref{cm} to $\overline{\varrho}$ solution of \eqref{varrhob}, we obtain
	\begin{eqnarray}\label{cmc}
	\mathcal{K}(\overline{\varrho})
	\leq  C\int_0^T\int_{\O^\prime}(\varrho^2+\varrho_{x}^2)e^{2s\Phi}\,\dq+C\int_0^T\int_{\O_1}s^3\Theta^3\varrho^2 e^{2s\Phi}\,\dq\nonumber\\
	+C\int_{Q}\left(\eta \overline{\rho}_{x}^2+s^2\eta^3\overline{\rho}^2\right)e^{2s\Phi}\ \dq+C\int_{Q}\eta\overline{\varrho}_x^2e^{2s\Phi}\ \dq.
\end{eqnarray}
	Combining \eqref{cm}-\eqref{cmc}, we obtain for $s$ large enough 
	
	\begin{eqnarray*}
		\dis \mathcal{K}(\overline{\rho})+\mathcal{K}(\overline{\varrho}) \leq C\int_0^T\int_{\O^\prime}(\rho^2+\varrho^2+\rho^2_{x}+\varrho^2_{x})e^{2s\Phi}\,\dq
		+C\int_0^T\int_{\O_1}s^3\Theta^3(\rho^2+\varrho^2)e^{2s\Phi}\,\dq.
	\end{eqnarray*}
	Using Caccioppoli's inequality \eqref{caccio}, this latter estimate becomes
	\begin{eqnarray}\label{cmd}
		\dis \mathcal{K}(\overline{\rho})+\mathcal{K}(\overline{\varrho}) \leq C\int_0^T\int_{\O_1}s^3\Theta^3(\rho^2+\varrho^2)e^{2s\Phi}\,\dq.
	\end{eqnarray}	
Thanks to \eqref{ineqphi}, the fact that $a\in \mathcal{C}^1([\alpha,1])$ and the function $\dis x\longmapsto\frac{x^2}{a(x)}$ is non-decreasing on $(0,1]$, one can prove the existence of a constant $C>0$  such that for all $(t,x)\in (0,T)\times [\alpha,1]$, we have
\begin{equation}\label{est}
\begin{array}{llll}
\dis e^{2s\varphi}\leq e^{2s\Phi},\ \ \ \, \frac{x^2}{a(x)}e^{2s\varphi}\leq Ce^{2s\Phi},\ \ \ \, a(x)e^{2s\varphi}\leq Ce^{2s\Phi}.
\end{array}
\end{equation}
Using \eqref{est}, the inequality \eqref{cmd} becomes
	
	\begin{eqnarray}\label{i17}
	\dis \mathcal{I}(\overline{\rho})+\mathcal{I}(\overline{\varrho}) \leq C\int_0^T\int_{\O_1}s^3\Theta^3(\rho^2+\varrho^2)e^{2s\Phi}\,\dq.
\end{eqnarray}	
Combining \eqref{i17} with \eqref{i5}, and using the fact that $e^{2s\varphi}\leq e^{2s\Phi}$, we obtain	
	\begin{eqnarray}\label{i18}
	\dis \mathcal{I}(\overline{\rho}+\widetilde{\rho})+\mathcal{I}(\overline{\varrho}+\widetilde{\varrho}) \leq C\int_0^T\int_{\O_1}s^3\Theta^3(\rho^2+\varrho^2)e^{2s\Phi}\,\dq.
\end{eqnarray}		
Using the fact that $\varrho=\widetilde{\varrho}+\overline{\varrho}$ and $\rho=\widetilde{\rho}+\overline{\rho}$, then we have 
\begin{equation}\label{pat}
	\begin{array}{rll}
	\dis	|\varrho|^2\leq 2\left(|\widetilde{\varrho}|^2+|\overline{\varrho}|^2\right),\ |\rho|^2\leq 2\left(|\widetilde{\rho}|^2+|\overline{\rho}|^2\right),\\
		\\
		 \dis |\varrho_{x}|^2\leq 2\left(|\widetilde{\varrho}_{x}|^2+|\overline{\varrho}_{x}|^2\right),\ |\rho_{x}|^2\leq 2\left(|\widetilde{\rho}_{x}|^2+|\overline{\rho}_{x}|^2\right). 
	\end{array}
\end{equation}
Combining \eqref{i18} with \eqref{pat}, we obtain the existence of a constant $C_1>0$ such that	
\begin{eqnarray*}
	\mathcal{I}(\rho)+\mathcal{I}(\varrho) \leq C_1\int_{0}^{T}\int_{\O_1}s^3\Theta^3(\rho^2+\varrho^2)e^{2s\Phi}\,\dq.
\end{eqnarray*}
This completes the proof of Theorem \ref{thm1}.
\end{proof}

\subsection{An observability inequality result}

This part is devoted to the observability inequality of systems \eqref{rhobis}-\eqref{varrho}. This inequality is obtained by using the intermediate Carleman estimate \eqref{ineqcarlprinc}.

\begin{prop} \label{prop5}  $ $
	
	Under the assumptions of Theorem \ref{thm1}, there exists a constant $C_2>0$, such that every solution $(\rho, \varrho)$ to \eqref{rhobis}-\eqref{varrho}, satisfies, for $s$ large enough, the following inequality:
		
	\begin{eqnarray}\label{obser}
		\mathcal{I}(\rho)+\mathcal{I}(\varrho) \leq C_2s^{7}\int_{0}^{T}\int_{\omega}|\rho|^2\,\dq,
	\end{eqnarray}
	where the notation $\mathcal{I}(\cdot)$ is defined by \eqref{I}.
	
\end{prop}

\begin{proof}
		
	We want to eliminate the local term corresponding to $\varrho$ on the right hand side of \eqref{ineqcarlprinc}.
	\noindent So, let $\O_2$ be a non empty set such that $\O_1\Subset \O_2\Subset \omega_d\cap\omega$. We introduce as in \cite{teresa2000insensitizing} the cut off function $\xi\in C^{\infty}_0(\Omega)$ such that 
	\begin{subequations}\label{owogene1}
		\begin{alignat}{11}
		\dis 	0\leq \xi\leq 1\ \mbox{in}\ \Omega, \,\,\xi=1 \hbox{ in } \O_1 ,\,\,  \xi=0 \hbox{ in } \Omega\setminus\O_2,\\
		\dis 	\frac{\xi_{xx}}{\xi^{1/2}}\in L^{\infty}(\O_2),\,\, \frac{\xi_{x}}{\xi^{1/2}}\in L^{\infty}(\O_2).
		\end{alignat}
	\end{subequations}
Set $\dis u=s^3\Theta^3e^{2s\Phi}$. Then $u(T)=u(0)=0$ and we have the following estimations:

\begin{equation}\label{con}
	\begin{array}{rll}
		\dis |u\xi|\leq s^3\Theta^3e^{2s\Phi}\xi,\ \ \ \ \ \
		\dis \left|(u\xi)_t\right|\leq Cs^4\Theta^8e^{2s\Phi}\xi,\\
		\\
		\dis |(u\xi)_x|\leq Cs^4\Theta^4e^{2s\Phi}\xi,\ \ \ \ \ \
		\dis |(a(x)(u\xi)_x)_x|\leq Cs^5\Theta^5e^{2s\Phi}\xi,
	\end{array}
\end{equation}
where $C$ is a positive constant.\\
	Multiplying the first equation of \eqref{rhobis} by $u\xi\varrho$ and integrating by parts over $Q$, we obtain
	\begin{equation}\label{beau}
		J_1+J_2+J_3+J_4+J_5+J_6+J_7=\int_{Q}u\xi|\varrho|^2\chi_{\omega_d}\ \dT,
	\end{equation}
where
	$$
	\begin{array}{lllll}
	J_1=\dis-\frac{\alpha_1}{\mu_1}\int_{Q}u\xi|\rho|^2\chi_{\omega_{1}} \ \dT-\frac{\alpha_2}{\mu_2}\int_{Q}u\xi|\rho|^2\chi_{\omega_{2}} \ \dT,\,
	J_2=\int_{Q}\rho\varrho\frac{\partial (u\xi)}{\partial t}\ \dT,\\
	  J_3= \dis -\int_{Q}(a(x)(u\xi)_x)_x\ \rho\varrho\ \dT,\ 
	\dis J_4=\dis-2\int_Qa(x)(u\xi)_x\rho\varrho_{x}\ \dT,\\ J_5=\dis\int_Q(a_1-b_1)u\xi\rho\varrho\ \dT,\ J_6=\dis\int_Q\beta(x)u\xi\rho(a_2-b_2)\varrho_{x}\ \dT,\ J_7=\dis\int_Q\beta(x)a_2(u\xi)_x\rho\varrho\ \dT.
	\end{array}
	$$
		
	Let us estimate each $J_i,\ i=1,\cdots,7$. From Young's inequality, we have
	
	\begin{eqnarray*}
		J_1&\leq& \dis \left(\frac{\alpha_1}{\mu_1}+\frac{\alpha_2}{\mu_2}\right)C\int_{Q} s^3\Theta^3e^{2s\Phi}\xi|\rho|^2\ \dT\\
		&\leq&\dis \left(\frac{\alpha_1^2}{\mu_1^2}+\frac{\alpha_2^2}{\mu_2^2}\right)\int_Q s^3\Theta^3\frac{x^2}{a(x)}e^{2s\varphi}|\rho|^2\dT+C\int_{0}^{T}\int_{\O_2} s^3\Theta^3\frac{a(x)}{x^2}e^{2s(2\Phi-\varphi)}|\rho|^2\dT,	
	\end{eqnarray*}
	
	\begin{eqnarray*}
		J_2&\leq&\dis C\int_{Q}s^4\Theta^8e^{2s\Phi}\xi|\rho\varrho|\ \dT\\
		&\leq&\dis \frac{\delta_1}{2}\int_Q s^3\Theta^3\frac{x^2}{a(x)}e^{2s\varphi}|\varrho|^2\dT+C_{\delta_1}\int_{0}^{T}\int_{\O_2} s^5\Theta^{13}\frac{a(x)}{x^2}e^{2s(2\Phi-\varphi)}|\rho|^2\dT,
	\end{eqnarray*}
	for any $\delta_1>$0.
	
	\begin{eqnarray*}
		J_3&\leq&C\int_{Q}s^5\Theta^5e^{2s\Phi}\xi|\rho\varrho|\ \dT\\
		&\leq&\dis \frac{\delta_2}{2}\int_Q s^3\Theta^3\frac{x^2}{a(x)}e^{2s\varphi}|\varrho|^2\dT+C_{\delta_2}\int_{0}^{T}\int_{\O_2} s^7\Theta^7\frac{a(x)}{x^2}e^{2s(2\Phi-\varphi)}|\rho|^2\dT,
	\end{eqnarray*}
for any $\delta_2>$0.

		\begin{eqnarray*}
			J_4&\leq&C\int_{Q}s^4\Theta^4a(x)e^{2s\Phi}\xi|\rho\varrho_{x}|\ \dT\\
			&\leq&\dis \frac{\delta_3}{2}\int_Q s\Theta a(x)e^{2s\varphi}|\varrho_{x}|^2\dT+C_{\delta_3}\int_{0}^{T}\int_{\O_2} s^7\Theta^7a(x)e^{2s(2\Phi-\varphi)}|\rho|^2\dT,
		\end{eqnarray*}
	for some $\delta_3>$0.
	
		\begin{eqnarray*}
		J_5&\leq&C\int_{Q}(a_1-b_1)s^3\Theta^3e^{2s\Phi}\xi|\rho\varrho|\ \dT\\
		&\leq&\dis \frac{\delta_4}{2}\int_Q s^3\Theta^3 \frac{x^2}{a(x)}e^{2s\varphi}|\varrho|^2\dT+C_{\delta_4}(\|a_1-b_1\|_{\infty})\int_{0}^{T}\int_{\O_2} s^3\Theta^3\frac{a(x)}{x^2}e^{2s(2\Phi-\varphi)}|\rho|^2\dT,
	\end{eqnarray*}
for some $\delta_4>$0.

	\begin{eqnarray*}
	J_6&\leq&C\int_{Q}\beta(x)(a_2-b_2)s^3\Theta^3e^{2s\Phi}\xi|\rho\varrho_x|\ \dT\\
	&\leq&\dis \frac{\delta_5}{2}\int_Q s\Theta a(x)e^{2s\varphi}|\varrho_x|^2\dT+C_{\delta_5}(\|a_2-b_2\|_{\infty})\int_{0}^{T}\int_{\O_2} s^5\Theta^5\frac{1}{a(x)}\beta^2(x)e^{2s(2\Phi-\varphi)}|\rho|^2\dT.
\end{eqnarray*}
Since $\dis |\beta(x)|\leq L\sqrt{a(x)}$, we have
	\begin{eqnarray*}
	J_6
	\leq\dis \frac{\delta_5}{2}\int_Q s\Theta a(x)e^{2s\varphi}|\varrho_x|^2\dT+C_{\delta_5}(\|a_2-b_2\|_{\infty})\int_{0}^{T}\int_{\O_2} s^5\Theta^5e^{2s(2\Phi-\varphi)}|\rho|^2\dT,
\end{eqnarray*}
for any $\delta_5>$0.

	\begin{eqnarray*}
	J_7&\leq&C\int_{Q}\beta(x)\beta (x)a_2s^4\Theta^4e^{2s\Phi}\xi|\rho\varrho|\ \dT\\
	&\leq&\dis \frac{\delta_6}{2}\int_Q s^3\Theta^3 \frac{x^2}{a(x)}e^{2s\varphi}|\varrho|^2\dT+C_{\delta_6}(\|a_2\|_{\infty})\int_{0}^{T}\int_{\O_2} s^5\Theta^5\frac{a(x)}{x^2}\beta^2(x)e^{2s(2\Phi-\varphi)}|\rho|^2\dT.
\end{eqnarray*}
Since $\dis |\beta(x)|\leq L\sqrt{a(x)}$, we have
\begin{eqnarray*}
	J_7
	\leq\dis \frac{\delta_6}{2}\int_Q s^3\Theta^3 \frac{x^2}{a(x)}e^{2s\varphi}|\varrho|^2\dT+C_{\delta_6}(\|a_2-b_2\|_{\infty})\int_{0}^{T}\int_{\O_2} s^5\Theta^5\frac{a^2(x)}{x^2}e^{2s(2\Phi-\varphi)}|\rho|^2\dT,
\end{eqnarray*}
for any $\delta_6>$0.
	
	Finally, choosing the constants $\delta_i$ such that $\dis \delta_1=\delta_2=\delta_4=\delta_6=\frac{1}{4C_1}$ and $\dis \delta_3=\delta_5=\frac{1}{2C_1}$, where $C_1$ is the constant obtained to Theorem \ref{thm1}, it follows from \eqref{beau} and the previous inequalities that
	\begin{equation}\label{owo4}
	\begin{array}{rll}
	&&\dis \int_0^T\int_{\O_1}s^3\Theta^3e^{2s\Phi}|\varrho|^2\ \dT	
	\leq\dis \frac{1}{2C_1}\mathcal{I}(\varrho)+\dis \left(\frac{\alpha_1^2}{\mu_1^2}+\frac{\alpha_2^2}{\mu_2^2}\right)\int_0^T\int_{Q} s^3\Theta^3\frac{x^2}{a(x)}e^{2s\varphi}|\rho|^2\ \dT\\
	&&+\dis C(\|a_1-b_1\|_{\infty})\int_{0}^{T}\int_{\O_2} s^7\Theta^{13}\frac{a(x)}{x^2}e^{2s(2\Phi-\varphi)}|\rho|^2\ \dT
	\dis +C\int_{0}^{T}\int_{\O_2} s^7\Theta^7a(x)e^{2s(2\Phi-\varphi)}|\rho|^2\ \dT\\
	&&+\dis C(\|a_2-b_2\|_{\infty})\int_{0}^{T}\int_{\O_2} s^5\Theta^{5}e^{2s(2\Phi-\varphi)}|\rho|^2\ \dT
	\dis +C(\|a_2\|_{\infty})\int_{0}^{T}\int_{\O_2} s^5\Theta^5\frac{a^2(x)}{x^2}e^{2s(2\Phi-\varphi)}|\rho|^2\ \dT.
	\end{array}
	\end{equation}
	Combining \eqref{ineqcarlprinc} with \eqref{owo4} and taking $\mu_i,\ i=1,2$ large enough, we obtain 
	\begin{equation}\label{obser2}
		\begin{array}{rll}
			&&\dis \mathcal{I}(\rho)+\mathcal{I}(\varrho)	
			\leq
		\dis C(\|a_1-b_1\|_{\infty})\int_{0}^{T}\int_{\O_2} s^7\Theta^{13}\frac{a(x)}{x^2}e^{2s(2\Phi-\varphi)}|\rho|^2\ \dT\\
			&&\dis +C\int_{0}^{T}\int_{\O_2} s^7\Theta^7a(x)e^{2s(2\Phi-\varphi)}|\rho|^2\ \dT
			+\dis C(\|a_2-b_2\|_{\infty})\int_{0}^{T}\int_{\O_2} s^5\Theta^{5}e^{2s(2\Phi-\varphi)}|\rho|^2\ \dT\\
			&&\dis +C(\|a_2\|_{\infty})\int_{0}^{T}\int_{\O_2} s^5\Theta^5\frac{a^2(x)}{x^2}e^{2s(2\Phi-\varphi)}|\rho|^2\ \dT.
		\end{array}
	\end{equation}
Note that $\dis \frac{a(x)}{x^2}$, $\dis \frac{a^2(x)}{x^2}$ and $a(x)$ are bounded on $\O_2$. 
Then, using \eqref{obser2}, we obtain the existence of a positive constant $C_2$ such that 	
		\begin{eqnarray}\label{obser1}
		\dis \mathcal{I}(\rho)+\mathcal{I}(\varrho)
		\leq C_2\int_{0}^{T}\int_{\O_2} s^7\Theta^{13}e^{2s(2\Phi-\varphi)}|\rho|^2\ \dT.
		\end{eqnarray} 	
Thanks to \eqref{ineqphi}, we have $2\Phi-\varphi\leq 0$, then $\dis \Theta^{13}e^{2s(2\Phi-\varphi)}\in L^\infty(Q)$. Furthermore, using the fact that, $\O_2\subset\omega$, we deduce the inequality \eqref{obser} and we complete the proof of Proposition \ref{prop5}.	
\end{proof}
\paragraph{}

To prove the needed observability inequality, we are going to improve the Carleman inequality \eqref{obser} in the sense that the weight functions do not vanish at $t=0$. To this end, we modify the weight functions $\varphi$ and $\Theta$ defined in \eqref{functcarl} as follows:
\begin{equation}\label{phitil}
\widetilde{\varphi}(t,x)=\left\{
\begin{array}{rllll}
\dis\varphi\left(\frac{T}{2},x\right)\ \ \mbox{if}\ \ t\in \left[0,\frac{T}{2}\right],\\
\dis \varphi(t,x)\ \ \mbox{if}\ \ t\in \left[\frac{T}{2},T\right],
\end{array}
\right.
\end{equation}
and
\begin{equation}\label{Thetatil}
\widetilde{\Theta}(t)=\left\{
\begin{array}{rllll}
\dis\Theta\left(\frac{T}{2}\right)\ \ \mbox{if}\ \ t\in \left[0,\frac{T}{2}\right],\\
\dis \Theta(t)\ \ \mbox{if}\ \ t\in \left[\frac{T}{2},T\right].
\end{array}
\right.
\end{equation} Then in view of the definition of $\varphi$ and $\Theta$, the functions $\widetilde{\varphi}(.,x)$ and $\widetilde{\Theta}(\cdot)$ are of class $\mathcal{C}^1$ on $[0,T[$. We have the following result.

\begin{prop} \label{pro}$ $
	
	  Under the assumptions of Proposition \ref{prop5}, there exist a positive constant\\ $C=C(C_2,\|a_1\|_{\infty}, \|a_2\|_{\infty},\|b_1\|_{\infty}, \|b_2\|_{\infty},\mu_1,\mu_2, T)>0$ and a positive weight function $\kappa$ such that every solution $(\rho, \psi^i)$  of \eqref{rho}-\eqref{psi}, satisfies the following inequality:
	
	\begin{eqnarray}\label{obser3}
	\|\rho(0,\cdot)\|^2_{L^2(\Omega)}+\sum_{i=1}^{2}\int_{Q}\kappa^2|\psi^i|^2\,\dT
	\leq C\int_{0}^{T}\int_{\omega}|\rho|^2\,\dq,
	\end{eqnarray}
	where the constant $C_2$ is given by the Proposition \ref{prop5}.
\end{prop}

\begin{proof} 
	
	We proceed in two steps.\\
\textbf{Step 1.} We prove that there exist a constant $C=C(C_2,\|a_1\|_{\infty}, \|a_2\|_{\infty},\|b_1\|_{\infty}, \|b_2\|_{\infty},T)>0$ such that 
\begin{equation}\label{observ}
	\begin{array}{llll}
		\dis  \|\rho(0,\cdot)\|^2_{L^2(\Omega)}+\widetilde{\mathcal{I}}_{[0,T]}(\rho)+\widetilde{\mathcal{I}}_{[0,T]}(\varrho)
		\dis \leq\dis C\int_{0}^{T}\int_{\omega}|\rho|^2\,\dq,
	\end{array}
\end{equation}
	where $\widetilde{\mathcal{I}}(\cdot)$ is defined in the follow by \eqref{car3}.

	Following the strategy in \cite{araruna2015anash}, let us introduce a function $\zeta\in \mathcal{C}^1 ([0,T])$  such that
	\begin{equation}\label{hypobeta}
	0\leq \zeta\leq 1,\ \zeta(t)=1 \hbox{ for } t\in [0,T/2],\  \zeta(t)=0 \hbox{ for } t\in [3T/4,T],\  |\zeta^\prime(t)|\leq C/T.
	\end{equation}
	For any $(t,x)\in Q$, we set
	$$
	\begin{array}{lll}
	z(t,x)=\zeta(t)e^{-r(T-t)}\rho(t,x),
	\end{array}
	$$
	where $r>0$. Then in view of \eqref{rhobis}, the function $z$ is solution of
	\begin{equation}\label{eq26z}
		\left\{
		\begin{array}{rllll}
			\dis -z_{t}-\left(a(x)z_{x}\right)_{x}+a_1z-(\beta(x)a_2z)_x+rz  &=& \dis \zeta e^{-r(T-t)}\varrho\chi_{\omega_d}-\zeta^\prime e^{-r(T-t)} \rho & \mbox{in}& Q,\\
			\dis z(t,0)=z(t,1)&=&0& \mbox{on}& (0,T), \\
			\dis z(T,\cdot)&=&0 &\mbox{in}&\Omega.
		\end{array}
		\right.
	\end{equation}
	From the classical energy estimate for the system \eqref{eq26z} and using the definition of $\beta$ and $z$, there exists a positive constant $C=C(\|a_1\|_{\infty}, \|a_2\|_{\infty},T)$ such that
	$$
	\begin{array}{llll}
	\dis \|\rho(0,\cdot)\|^2_{L^2(\Omega)}+\int_{0}^{T/2}\int_{\Omega}|\rho|^2\ dxdt+\int_{0}^{T/2}\int_{\Omega}a(x)|\rho_{x}|^2\ dxdt\\
	\dis \leq C
	\left(\int_0^{3T/4}\int_\Omega |\varrho|^2\ \dq
	+\int_{T/2}^{3T/4}\int_\Omega|\rho|^2\ \dq\right).
	\end{array}
	$$
	The functions $\widetilde{\varphi}$ and $\widetilde{\Theta}$ defined by \eqref{phitil} and \eqref{Thetatil}, respectively, have lower and upper bounds for $(t,x)\in  [0,T/2]\times \Omega$. Furthermore, $\dis a(x)\geq c>0$ in $(0,1]$ and $\dis \frac{x^2}{a(x)}\geq c>0$ in $(0,1]$, then, we can introduce the corresponding weight functions in the above expression and we get
	\begin{equation}\label{pou}
	\begin{array}{llll}
	\dis \|\rho(0,\cdot)\|^2_{L^2(\Omega)}+\widetilde{\mathcal{I}}_{[0,T/2]}(\rho) \dis \leq C(\|a_1\|_{\infty}, \|a_2\|_{\infty},T)
	\left(\int_0^{3T/4}\int_\Omega |\varrho|^2\ \dq
	+\int_{T/2}^{3T/4}\int_\Omega|\rho|^2\ \dq\right),
	\end{array}
	\end{equation}
	where 
	\begin{equation}\label{car3}
	\widetilde{\mathcal{I}}_{[a,b]}(l)=\int_a^b\int_{\Omega}\widetilde{\Theta}^3\frac{x^2}{a(x)}e^{2s\widetilde{\varphi}}|l|^2\ \dq+\int_a^b\int_{\Omega}\widetilde{\Theta}a(x)e^{2s\widetilde{\varphi}}|l_x|^2\ \dq.
	\end{equation}
	Adding the term $\widetilde{\mathcal{I}}_{[0,T/2]}(\varrho)$ on both sides of inequality \eqref{pou}, we have
		\begin{equation}\label{pou1}
		\begin{array}{llll}
			\dis \|\rho(0,\cdot)\|^2_{L^2(\Omega)}+\widetilde{\mathcal{I}}_{[0,T/2]}(\rho)+\widetilde{\mathcal{I}}_{[0,T/2]}(\varrho)\\ 
			\dis \leq C(\|a_1\|_{\infty}, \|a_2\|_{\infty}, T)
			\left(\int_0^{3T/4}\int_\Omega |\varrho|^2\ \dq
			+\int_{T/2}^{3T/4}\int_\Omega |\rho|^2\ \dq\right)+\widetilde{\mathcal{I}}_{[0,T/2]}(\varrho).
		\end{array}
	\end{equation}
	In order to eliminate the term $\widetilde{\mathcal{I}}_{[0,T/2]}(\varrho)$ in the right hand side of \eqref{pou1}, we use the classical energy estimates for the system \eqref{varrho} and we obtain:
	$$
	\begin{array}{llll}
		\dis \int_{0}^{T/2}\int_{\Omega}|\varrho|^2\ dxdt+\int_{0}^{T/2}\int_{\Omega}a(x)|\varrho_{x}|^2\ dxdt\\
		\dis \leq C(\|b_1\|_{\infty}, \|b_2\|_{\infty},T) \left(\frac{\alpha_1^2}{\mu_1^2}+\frac{\alpha_2^2}{\mu_2^2}\right)\int_{0}^{T/2}\int_\Omega |\rho|^2\ dxdt,
	\end{array}
	$$
	where $C$ is independent of $\mu_i,\ i=1,2$. The functions $\widetilde{\varphi}$ and $\widetilde{\Theta}$ have lower and upper bounds for $(t,x)\in  [0,T/2]\times \Omega$. Moreover, the function $\dis x\longmapsto \frac{x^2}{a(x)}$ is non-decreasing on $(0;1]$ and $\dis \frac{x^2}{a(x)}\geq c>0$ in $(0,1]$. Then, from the previous inequality, we obtain 
	\begin{equation}\label{pou2}
	\begin{array}{llll}
	\dis \widetilde{\mathcal{I}}_{[0,T/2]}(\varrho)
	\dis \leq C(\|b_1\|_{\infty)}, \|b_2\|_{\infty},T) \left(\frac{\alpha_1^2}{\mu_1^2}+\frac{\alpha_2^2}{\mu_2^2}\right)\int_{0}^{T/2}\int_\Omega \widetilde{\Theta}^3\frac{x^2}{a(x)}e^{2s\widetilde{\varphi}}|\rho|^2\ dxdt.
     \end{array}
	\end{equation}
	Replacing \eqref{pou2} in \eqref{pou1} and taking $\mu_i,\ i=1,2$ large enough, we obtain
	\begin{equation}\label{pou3}
		\begin{array}{llll}
			\dis \|\rho(0,\cdot)\|^2_{L^2(\Omega)}+\widetilde{\mathcal{I}}_{[0,T/2]}(\rho)+\widetilde{\mathcal{I}}_{[0,T/2]}(\varrho)\\ 
			\dis \leq C(\|a_1\|_{\infty}, \|a_2\|_{\infty},\|b_1\|_{\infty}, \|b_2\|_{\infty},T)
		\int_{T/2}^{3T/4}\int_\Omega ( |\rho|^2+|\varrho|^2)\ \dq.
		\end{array}
	\end{equation}
	The functions $\varphi$ and $\Theta$ defined in \eqref{functcarl} have the lower and upper bounds for $(t,x)\in  [T/2,3T/4]\times \Omega$. Moreover, the function $\dis x\longmapsto\frac{x^2}{a(x)}$ is non-decreasing on $(0,1]$. Using the inequality \eqref{obser}, the relation \eqref{pou3} becomes
	\begin{equation}\label{pou4}
		\begin{array}{llll}
			&&\dis \|\rho(0,\cdot)\|^2_{L^2(\Omega)}+\widetilde{\mathcal{I}}_{[0,T/2]}(\rho)+\widetilde{\mathcal{I}}_{[0,T/2]}(\varrho)\\ 
			&&\dis \leq\dis  C(\|a_1\|_{\infty}, \|a_2\|_{\infty},\|b_1\|_{\infty}, \|b_2\|_{\infty},T)
			\left(\mathcal{I}(\rho)+\mathcal{I}(\varrho)\right)\\
			&&\leq\dis C(C_2,\|a_1\|_{\infty}, \|a_2\|_{\infty},\|b_1\|_{\infty}, \|b_2\|_{\infty},T)\int_{0}^{T}\int_{\omega}|\rho|^2\,\dq,
		\end{array}
	\end{equation}
	where $\mathcal{I}(\cdot)$ is defined by \eqref{I} and the constant $C_2$ is defined in the Proposition \ref{prop5}.
	
	On the other hand, since $\Theta=\widetilde{\Theta}$ and $\varphi=\tilde{\varphi}$ in $[T/2,T]\times \Omega$, we use again estimate \eqref{obser} and we obtain
	\begin{equation}\label{pou5}
	\begin{array}{llll}
	\dis \widetilde{\mathcal{I}}_{[T/2,T]}(\rho)+\widetilde{\mathcal{I}}_{[T/2,T]}(\varrho)&\leq&\dis \mathcal{I}(\rho)+\mathcal{I}(\varrho)\\
	&\leq&\dis C(C_2,\|a_1\|_{\infty}, \|a_2\|_{\infty},\|b_1\|_{\infty}, \|b_2\|_{\infty},T)\int_{0}^{T}\int_{\omega}|\rho|^2\,\dq.
	\end{array}
	\end{equation}
	Adding \eqref{pou4} and \eqref{pou5}, we get
	\begin{equation*}\label{}
		\begin{array}{llll}
			\dis  \|\rho(0,\cdot)\|^2_{L^2(\Omega)}+\widetilde{\mathcal{I}}_{[0,T]}(\rho)+\widetilde{\mathcal{I}}_{[0,T]}(\varrho)
			\dis \leq\dis C\int_{0}^{T}\int_{\omega}|\rho|^2\,\dq,
		\end{array}
	\end{equation*}
	 where $C=C(C_2,\|a_1\|_{\infty}, \|a_2\|_{\infty},\|b_1\|_{\infty}, \|b_2\|_{\infty},T)$ and then, we deduce the estimation \eqref{observ}.\\
	\textbf{Step 2.}
	Now, we prove that there exist a constant $C=C(C_2,\|a_1\|_{\infty}, \|a_2\|_{\infty},\|b_1\|_{\infty}, \|b_2\|_{\infty},T,\mu_1,\mu_2)>0$ and a positive weight function $\kappa$ such that
		\begin{eqnarray}\label{observ1}
		\sum_{i=1}^{2}\int_{Q}\kappa^2|\psi^i|^2\,\dT
		\leq C\int_{0}^{T}\int_{\omega}|\rho|^2\,\dq.
	\end{eqnarray}
	
	Let us introduce the function
	\begin{eqnarray}\label{berlin}
	\dis \hat{\varphi}(t)=\min_{x\in \Omega}\widetilde{\varphi}(t,x).
	\end{eqnarray}
 We set the parameter $s=\bar{s}$ to a fixed value sufficiently
 large and we define the weight function $\kappa$  by:

\begin{eqnarray}\label{deftkappa}
	\dis \kappa(t)=e^{\bar{s}\hat{\varphi}(t)}\in L^\infty(0,T).
\end{eqnarray}
	Then $\kappa$ is a strictly positive function of class $\mathcal{C}^1$ on $[0,T)$ blowing up at $t=T$. Furthermore, $\dis \frac{\partial \hat{\varphi}}{\partial t}$ is also a positive function on $(0,T)$. 
	Now, multiplying the first equation of \eqref{psi} by $\kappa^2\psi^i$ and integrating by parts over $\Omega$, we obtain that 
	\begin{equation}\label{bec}
	\begin{array}{rlll}
	\dis \frac{1}{2}\frac{d}{dt}\int_{\Omega} \kappa^2|\psi^i|^2\ dx+\int_{\Omega} \kappa^2a(x)|\psi_{x}^i|^2\ dx
	\dis &=&\dis -\int_{\Omega}\kappa^2\, b_1\,|\psi^i|^2\ dx-\int_{\Omega}\kappa^2\, |\psi^i|^2\beta(x)b_2\psi_x^i\ dx\\
	&&-\dis \frac{1}{\mu_i}\int_{\omega_i} \kappa^2\rho\psi^i\ dx+\bar{s}\dis \int_{\Omega} \kappa^2\frac{\partial\hat{\varphi}}{\partial t}|\psi^i|^2\ dx.
	\end{array}
	\end{equation}
Since $\dis |\beta(x)|\leq L\sqrt{a(x)}$ and using the fact that $\dis \frac{\partial \hat{\varphi}}{\partial t}$ is a positive function on $[0,T)$, the inequality \eqref{bec} becomes
\begin{equation}\label{}
	\begin{array}{rlll}
		\dis \frac{1}{2}\frac{d}{dt}\int_{\Omega} \kappa^2|\psi^i|^2\ dx+\int_{\Omega} \kappa^2a(x)|\psi_{x}^i|^2\ dx
		\dis 
		\leq C\int_{\Omega} \kappa^2|\psi^i|^2\ dx+\frac{1}{2\mu_i^2}\int_{\omega_i} \kappa^2|\rho|^2\ dx
	\end{array}
\end{equation}
where $C=\left(\|b_1\|_{\infty}, \|b_2\|_{\infty}\right)=\dis \left(\|b_1\|_{\infty}+\frac{1}{2}\|b_2\|^2_{\infty}+\frac{1}{2}\right)$.
	Using Gronwall's Lemma and the fact that $\psi^i(x,0)=0$ for $x\in \Omega$, it follows that
	\begin{equation}\label{berlin2}
	\begin{array}{rlll}
	\dis \int_{\Omega} \kappa^2|\psi^i|^2\ dx\leq C\int_Q  \kappa^2|\rho|^2\ dx,\ \forall t\in [0,T],
	\end{array}
	\end{equation}
where $C=\left(\|b_1\|_{\infty}, \|b_2\|_{\infty}, T,\mu_1,\mu_2\right)$.
Using the definition of $\hat{\varphi}$ and $\kappa$ given by \eqref{berlin} and \eqref{deftkappa}, respectively, we have
\begin{equation}\label{ten}
	\kappa^2(t)\leq e^{2\bar{s}\widetilde{\varphi}(t,x)},\ \ \forall x\in \Omega.
\end{equation}
Thanks to the fact that $\dis \widetilde{\Theta}^{-1}\in L^\infty(0,T)$ and that the function $\dis x\longmapsto\frac{a(x)}{x^2}$ is non-decreasing on $(0,1]$, we deduce from \eqref{ten} the following inequality
	\begin{equation*}
	\begin{array}{rlll}
	\dis \int_{Q} \kappa^2|\psi^i|^2\ dx\leq \int_{Q} \widetilde{\Theta}^3\frac{x^2}{a(x)}e^{2s\widetilde{\varphi}}|\rho|^2\ \dT,
	\end{array}
	\end{equation*}
	which combining with \eqref{berlin2} and  \eqref{observ} give
	\begin{equation*}
	\begin{array}{rlll}
	\dis \int_{Q} \kappa^2|\psi^i|^2\ dx\leq C\int_0^T\int_{\omega} |\rho|^2 \dT,
	\end{array}
	\end{equation*}
	where $C=C(C_2,\|a_1\|_{\infty}, \|a_2\|_{\infty},\|b_1\|_{\infty}, \|b_2\|_{\infty},T,\mu_1,\mu_2)>0$.
	Adding this latter inequality with \eqref{observ}, we deduce \eqref{obser3}.	This ends the proof.
\end{proof}

\section{Null controllability of semilinear degenerate system}\label{null}	

In this section, we end the proof of Theorem \ref{theolinear}. More precisely, we prove that the linear systems \eqref{yli}-\eqref{pli} are null controllable.  Thanks to the Proposition \ref{pro}, the following result holds.

\begin{prop}\label{linear}$ $
	
	Suppose that \eqref{od} holds, $\mu_i,\ i=1,2$ are large enough, the coefficient $a(\cdot)$ satisfies \eqref{k}, $y^0\in L^2(\Omega)$ and $y_{i,d}\in L^2((0,T)\times \omega_{d})$ such that \eqref{nou} holds. Then, there exists a leader control $\bar{h}\in L^2(\omega_T)$ such that the corresponding solutions to \eqref{yli}-\eqref{pli} satisfies \eqref{mainobj}. Furthermore; there exists a constant $C=C(C_2,\|a_1\|_{\infty}, \|a_2\|_{\infty},\|b_1\|_{\infty}, \|b_2\|_{\infty},\mu_1,\mu_2, T)>0$ such that
	\begin{equation}\label{mon10theo}
		\begin{array}{ccc}
			\dis \|\bar{h}\|_{L^2(\omega_T)}\leq C
			\dis  \left(\sum_{i=1}^2\alpha_i^2\left\|\kappa^{-1} y_{i,d}\right\|^2_{L^2((0,T)\times\omega_d)}+ \|y^0\|^2_{L^2(\Omega)}\right)^{1/2}.
		\end{array}
	\end{equation}
\end{prop}

\begin{proof}

To prove this null controllability result, we proceed in three steps using a penalization method.\\
\noindent\textbf{Step 1.} For any fixed $\varepsilon >0$, we define a functional
\begin{equation}\label{defJ}
	J_{\varepsilon}(h)=\dis \frac{1}{2\varepsilon}\int_{\Omega}|y(T,\cdot)|^2\ dx+
	\frac{1}{2}\int_{\omega_T}|h|^2\ \dq,
\end{equation}
where $y$ is the solution to problem \eqref{yli}-\eqref{pli}.
Then we consider the optimal control problem: 
\begin{equation}\label{opt}
	\inf_{\atop h\in L^2(\omega_T)}J_{\varepsilon}(h).
\end{equation}
By standard arguments, we can prove that the functional $J_{\varepsilon}$ is continuous, coercive and strictly convex. Then, the optimization problem \eqref{opt} admits a unique solution $h_\varepsilon$  and arguing as in \cite{djomegne2021}, we prove that 
\begin{equation}\label{hgeps}
	h_{\varepsilon}=\rho_{\varepsilon}\ \ \mbox{in}\ \ \omega_T,
\end{equation}	
with $(\rho_\varepsilon,\psi^i_\varepsilon)$ is the solution of the following systems
\begin{equation}\label{rhoeps}
	\left\{
	\begin{array}{rllll}
		\dis -\rho_{\varepsilon,t}-\left(a(x)\rho_{\varepsilon,x}\right)_{x}+a_1\rho_{\varepsilon}-(\beta(x)a_2\rho_{\varepsilon})_x  &=&\dis (\alpha_1\psi_{\varepsilon}^1+\alpha_2\psi_{\varepsilon}^2)\chi_{\omega_{d}}& \mbox{in}& Q,\\
		\dis \rho_{\varepsilon}(t,0)=\rho_{\varepsilon}(t,1)&=&0& \mbox{on}& (0,T), \\
		\dis  \rho_{\varepsilon}(T,\cdot)&=&-\dis \frac{1}{\varepsilon}y_{\varepsilon} (T,\cdot) &\mbox{in}&\Omega
	\end{array}
	\right.
\end{equation}
and
\begin{equation}\label{psieps}
	\left\{
	\begin{array}{rllll}
		\dis \psi_{\varepsilon,t}^i-\left(a(x)\psi^i_{\varepsilon,x}\right)_{x}+b_1\psi_{\varepsilon}^i +\beta(x)b_2\psi^i_{\varepsilon,x}  &=&\dis-\frac{1}{\mu_i}\rho_{\varepsilon}\chi_{\omega_{i}}& \mbox{in}& Q,\\
		\dis \psi_{\varepsilon}^i(t,0)=\psi_{\varepsilon}^i(t,1)&=&0& \mbox{on}& (0,T), \\
		\dis  \psi_{\varepsilon}^i(0,\cdot)&=&0 &\mbox{in}&\Omega,
	\end{array}
	\right.
\end{equation}
where $(y_\varepsilon, p^i_\varepsilon)$ is solution of 
\begin{equation}\label{yeps}
	\left\{
	\begin{array}{rllll}
		\dis y_{\varepsilon,t}-\left(a(x)y_{\varepsilon,x}\right)_{x}+a_1y_{\varepsilon}+\beta(x)a_2y_{\varepsilon,x}  &=&\dis h_\varepsilon\chi_{\omega}-\frac{1}{\mu_1}p_{\varepsilon}^1\chi_{\omega_1}-\frac{1}{\mu_2}p_{\varepsilon}^2\chi_{\omega_2}& \mbox{in}& Q,\\
		\dis y_{\varepsilon}(t,0)=y_{\varepsilon}(t,1)&=&0& \mbox{on}& (0,T), \\
		\dis  y_{\varepsilon}(0,\cdot)&=&y^0&\mbox{in}&\Omega
	\end{array}
	\right.
\end{equation}
and 
\begin{equation}\label{peps}
	\left\{
	\begin{array}{rllll}
		\dis -p_{\varepsilon,t}^i-\left(a(x)p^i_{\varepsilon,x}\right)_{x}+b_1p_{\varepsilon}^i-(\beta(x)b_2p_{\varepsilon}^i)_x    &=&\alpha_i\left(y_{\varepsilon}-y_{i,d}\right)\chi_{\omega_{d}}& \mbox{in}& Q,\\
		\dis p_{\varepsilon}^i(t,0)=p_{\varepsilon}^i(t,1)&=&0& \mbox{on}& (0,T), \\
		\dis p_{\varepsilon}^i(T,\cdot)&=&0 &\mbox{in}&\Omega.
	\end{array}
	\right.
\end{equation}
\textbf{Step 2.} 
Multiplying  the first equation of \eqref{rhoeps} and \eqref{psieps}  by $y_{\varepsilon}$ and $p^i_{\varepsilon}$ respectively and integrating by parts over $Q$, we obtain from \eqref{hgeps}
$$
\begin{array}{rlll}
\dis \|h_\varepsilon\|^2_{L^2(\omega_T)}+\frac{1}{\varepsilon}\|y_{\varepsilon}(T,\cdot)\|^2_{L^2(\Omega)}
=\dis-\int_{\Omega}y^0\rho_\varepsilon(0,\cdot)\ dx+
\dis \sum_{i=1}^2\alpha_i\int_{0}^{T}\int_{\omega_d}y_{i,d}\psi^i_{\varepsilon} \dT.	
\end{array}
$$
Using the Young inequality, one can get that
\begin{equation}\label{rosnymon10}
	\begin{array}{ccc}
		\dis \|h_\varepsilon\|^2_{L^2(\omega_T)}+\frac{1}{\varepsilon}\|y_{\varepsilon}(T,\cdot)\|^2_{L^2(\Omega)}
		&\leq&
		\dis  \left(\sum_{i=1}^2\alpha_i^2\left\|\kappa^{-1} y_{i,d}\right\|^2_{L^2((0,T)\times\omega_d)}+ \|y^0\|^2_{L^2(\Omega)}\right)^{1/2}\\ 
		&&\dis \times\left( \sum_{i=1}^2 \left\|\kappa\psi_{\varepsilon}^i\right\|^2_{L^2(Q)}+
		\|\rho_{\varepsilon}(0,\cdot)\|^2_{L^2(\Omega)}\right)^{1/2}.
	\end{array}
\end{equation}
Using the observability inequality \eqref{obser3}, we deduce from \eqref{rosnymon10} the existence of a constant $C=C(C_2,\|a_1\|_{\infty}, \|a_2\|_{\infty},\|b_1\|_{\infty}, \|b_2\|_{\infty},\mu_1,\mu_2, T)>0$ such that
\begin{equation}\label{mon10}
	\begin{array}{ccc}
		\dis \|h_\varepsilon\|_{L^2(\omega_T)}&\leq&C
		\dis  \left(\sum_{i=1}^2\alpha_i^2\left\|\kappa^{-1} y_{i,d}\right\|^2_{L^2((0,T)\times\omega_d)}+ \|y^0\|^2_{L^2(\Omega)}\right)^{1/2}
	\end{array}
\end{equation}
and
\begin{equation}\label{mon10b}
	\begin{array}{ccc}
		\dis \|y_{\varepsilon}(T,\cdot)\|_{L^2(\Omega)}&\leq&
		\dis C\sqrt{\varepsilon} \left(\sum_{i=1}^2\alpha_i^2\left\|\kappa^{-1} y_{i,d}\right\|^2_{L^2((0,T)\times\omega_d)}+ \|y^0\|^2_{L^2(\Omega)}\right)^{1/2}.
	\end{array}
\end{equation}
Using \eqref{mon10}-\eqref{mon10b} and systems \eqref{yeps}-\eqref{peps}, we can extract subsequences still denoted by $h_\varepsilon,\ y_\varepsilon $ and $p^i_\varepsilon$ such that when $\varepsilon \rightarrow 0$, one has
\begin{subequations}\label{convergence2}
	\begin{alignat}{9}
		h_\varepsilon&\rightharpoonup& \bar{h}&\text{ weakly in }&L^{2}(\omega_T), \label{18k}\\
		y_{\varepsilon}&\rightharpoonup& y&\text{  weakly in }&L^2((0,T);H^1_a(\Omega)), \label{}\\
		p^i_{\varepsilon}&\rightharpoonup& p^i&\text{  weakly in }&L^2((0,T);H^1_a(\Omega)),\ i=1,2, \label{19}\\
		y_{\varepsilon}(T,\cdot)&\longrightarrow&0&\hbox{ strongly in }&\ L^2(\Omega).\label{all5}
	\end{alignat}
\end{subequations}

Arguing as in \cite{djomegne2018,djomegne2021}, using convergences \eqref{convergence2}, we prove that $(y,\ p^i)$ is a solution of \eqref{yli}-\eqref{pli} corresponding to the control $\bar{h}$ and also $y$ satisfies \eqref{mainobj}. Furthermore, using the convergence \eqref{18k}, we have that $\bar{h}$ satisfies \eqref{mon10theo}.
\end{proof}

\subsection{Proof of Theorem \ref{theolinear}}
In this subsection, we want to end the proof of Theorem \ref{theolinear}.
We have proved in Proposition \ref{quasi} and Theorem \ref{nash} that the Nash equilibrium for $(J_1,J_2)$ given by \eqref{all16}, $(\hat{v}_1, \hat{v}_2)$ is characterised by \eqref{vop}-\eqref{pop}. In Proposition \ref{linear}, we proved that the linear systems \eqref{yli}-\eqref{pli} is null controllable at time $t=T$. We are now going to prove that, there exists a control $\bar{h}\in L^2(\omega_T)$ such that the solution of \eqref{yop}-\eqref{pop} satisfies \eqref{mainobj}.

We define $Z=L^2((0,T);H^1_{a}(\Omega))$. We observe that, for any $y\in Z$, we have 
$$
F(y,y_x)-F(0,0)=F_1(y,y_x)y+F_2(y,y_x)y_x,
$$
where 
\begin{equation}\label{abbis}
	F_1(y,y_x)=\dis \int_0^1 D_1F(ry,ry_x)\ dr\ \mbox{and}\ F_2(y,y_x)=\dis \int_0^1 D_2F(ry,ry_x)\ dr.
\end{equation}
 For every $z\in Z$, we consider the linearized system for \eqref{yop}-\eqref{pop} 
 \begin{equation}\label{ylina}
 	\left\{
 	\begin{array}{rllll}
 		\dis y_{t}-\left(a(x)y_{x}\right)_{x}+F_1(z,z_x)y+\beta(x)\left[\frac{F_2(z,z_x)}{\beta(x)}\right]y_x  &=&\dis h\chi_{\omega}-\frac{1}{\mu_1}p^1\chi_{\omega_1}-\frac{1}{\mu_2}p^2\chi_{\omega_2}& \mbox{in}& Q,\\
 		\dis y(t,0)=y(t,1)&=&0& \mbox{on}& (0,T), \\
 		\dis y(0,\cdot)&=&y^0&\mbox{in}&\Omega
 	\end{array}
 	\right.
 \end{equation}
 and 
 \begin{equation}\label{plina}
 	\left\{
 	\begin{array}{rllll}
 		\dis -p_{t}^i-\left(a(x)p^i_{x}\right)_{x}+b_1p^i-(\beta(x)b_2p^i)_x   &=&\alpha_i\left(y-y_{i,d}\right)\chi_{\omega_{d}}& \mbox{in}& Q,\\
 		\dis p^i(t,0)=p^i(t,1)&=&0& \mbox{on}& (0,T), \\
 		\dis p^i(T,\cdot)&=&0 &\mbox{in}&\Omega.
 	\end{array}
 	\right.
 \end{equation}
 
Observe that systems \eqref{ylina}-\eqref{plina} are of the form \eqref{yli}-\eqref{pli} with 
 \begin{equation}\label{abc}
	\left\{
	\begin{array}{rllll}
	\dis a_1=a_1^z=F_1(z,z_x),\ \ a_2=a_2^z=\frac{F_2(z,z_x)}{\beta(x)},\\
	\dis 	b_1:=b_1^z=D_1F(z,z_x),\ \ b_2:=b_2^z=\frac{D_2F(z,z_x)}{\beta(x)}.
	\end{array}
	\right.
\end{equation}
Thanks to assumptions on $F$ (see $(H_5)$ and $(H_6)$ of Assumption \ref{assum}), there exists a positive constant $M$ such that
\begin{equation}
	\|a_1^z\|_{\infty},\ \|a_2^z\|_{\infty},\ \|b_1^z\|_{\infty},\ \|b_2^z\|_{\infty} \leq M,\ \forall z\in Z.
\end{equation}

 From Proposition \ref{linear}, there exits a control $\bar{h}^z\in L^2(\omega_T)$ such that the solution $(y^z, p^{i,z})$ to \eqref{ylina}-\eqref{plina} corresponding to coefficients given by \eqref{abc} and $\bar{h}=\bar{h}^z$ satisfies \eqref{mainobj},
and furthermore we have the estimate
\begin{equation}\label{mon10theoa}
	\begin{array}{ccc}
		\dis \|\bar{h}^z\|_{L^2(\omega_T)}\leq C
		\dis  \left(\sum_{i=1}^2\alpha_i^2\left\|\kappa^{-1} y_{i,d}\right\|^2_{L^2((0,T)\times\omega_d)}+ \|y^0\|^2_{L^2(\Omega)}\right)^{1/2},\ \forall z\in Z,
	\end{array}
\end{equation}
where
$C=C(C_2,\|a_1\|_{\infty}, \|a_2\|_{\infty},\|b_1\|_{\infty}, \|b_2\|_{\infty},\mu_1,\mu_2, T)>0$.

From Proposition \ref{exis}, we obtain that
$$
y^z, p^{i,z}\in Z,
$$
and using \eqref{mon10theoa} and \eqref{v}, we obtain
\begin{equation}\label{y}
	\dis \|y^z\|_{Z}+\|p^{i,z}\|_{Z}
	\leq 
	C\left(\sum_{i=1}^{2}\|y_{i,d}\|_{L^2((0,T)\times\omega_{d})}+
	\sum_{i=1}^2\alpha_i^2\left\|\kappa^{-1} y_{i,d}\right\|^2_{L^2((0,T)\times\omega_{d})}+\|y^0\|_{L^2(\Omega)}\right),
\end{equation} 	
where
$C=C(C_2,\|a_1\|_{\infty}, \|a_2\|_{\infty},\|b_1\|_{\infty}, \|b_2\|_{\infty},\mu_1,\mu_2,\alpha_1,\alpha_2, T)>0$.

For every $z\in Z$, we define

\begin{equation}
	I(z)=\left\{\bar{h}^z\in L^2(\omega_T),\ (y^z,p^{i,z})\ \mbox{solution of}\ \eqref{ylina}-\eqref{plina}\ \mbox{satisfies}\ \eqref{mainobj}\ \mbox{with}\ \bar{h}^z\ \mbox{verifying}\ \eqref{mon10theoa}\right\}	
\end{equation}
and 
\begin{equation}
	\Lambda(z)=\left\{(y^z,p^{i,z}):\ (y^z,p^{i,z})\ \mbox{is the state associated to a control}\ \bar{h}^z\in I(z)\ \mbox{and}\ (y^z,p^{i,z})\ \mbox{satisfies}\ \eqref{y}\right\}.	
\end{equation}
In this way, we introduce a multivalued mapping
$$
z\longmapsto \Lambda(z).
$$
We want to prove that this mapping has a fixed point $y$. Of course, this will imply that there exists a control $\bar{h}\in L^2(\omega_T)$ such that the solution of \eqref{yop}-\eqref{pop} satisfies \eqref{mainobj}.

To this end, we will use the Kakutani's fixed point Theorem that can be applied on $\Lambda$. Proceeding as in \cite{ danynina2021, birba2016}, we can prove the following properties for every $z\in Z$.

\begin{prop}$ $
	
	\begin{enumerate}
		\item $\Lambda(z)$ is a non empty, closed and convex set of $Z$.
		\item $\Lambda(z)$ is a bounded and compact set of $Z$.
		\item The application $z\longmapsto\Lambda(z)$ is upper hemi-continuous. 
	\end{enumerate}	
\end{prop}
This end the proof of Theorem \ref{theolinear} and furthermore the proof of null controllability of system \eqref{eq}.	\hfill $\blacksquare$

\section{Conclusion}  \label{conclusion}

In this paper, we established the Stackelberg-Nash null controllability of a semilinear degenerate parabolic equation with a non-linearity involving a gradient term. We proved that we can act on the system with one leader and two followers. Since the system is semilinear, the functionals are not convex in general. To overcome this difficulty, we first obtain  the existence and uniqueness of the Nash quasi-equilibrium, which is a weaker formulation of the Nash equilibrium. Next, with additional conditions, we established the equivalence between the Nash quasi-equilibrium and the Nash equilibrium. Finally, after establishing suitable Carleman estimates, we proved an observability inequality allowing us to obtain our null controllability result.


\section{Appendix}

\subsection{Proof of Proposition \ref{exis}}

\begin{proof}

We proceed in three steps.\\
\noindent \textbf{Step 1.} We show the estimate \eqref{esty1y2}.
Make the change of variable $z(t,x)=e^{-rt}y(t,x),\,\, (t,x)\in Q$,  for some $r>0$ where $y$ is solution to \eqref{ywell}. We obtain that $z$ is solution to
\begin{equation}\label{p1}
	\left\{
	\begin{array}{rllll}
		\dis z_t-(a(x)z_x)_x+a_0z+\beta(x)b_0z_x+r z &=&(h\chi_{\omega}+v^1\chi_{\omega_1}+v^2\chi_{\omega_2})e^{-rt} \qquad &\mbox{in}& Q,\\
		\dis z(t,0)=z(t,1)&=&0& \mbox{on}& (0,T), \\
		\dis z(0,\cdot)&=&y^0&\mbox{in}&\Omega.
	\end{array}
	\right.
\end{equation}
If we multiply the first equation in \eqref{p1} by $z$ and integrate by parts over $Q$, we obtain
\begin{equation*}
	\begin{array}{rll}
		\dis \int_{Q}z_tz\ dxdt-\int_{Q}(a(x)z_x)_x z\ dxdt+\int_{Q}rz^2\ dxdt=\dis -\int_{Q}a_0z^2\ dxdt-\int_{Q}\beta(x)b_0z_xz\ dxdt\\
		\dis+ \int_{Q}z(h\chi_{\omega}+v^1\chi_{\omega_1}+v^2\chi_{\omega_2})e^{-rt}\ dxdt.
	\end{array}
\end{equation*}
This latter equality becomes
\begin{equation}\label{estimation1A1}
	\begin{array}{rlll}
		&&\dis\frac{1}{2}\|z(T,\cdot)\|^2_{L^2(\Omega)}-\frac{1}{2}\|z(0,\cdot)\|^2_{L^2(\Omega)}+\|\sqrt{a(x)}z_x\|^2_{L^2(Q)}+r\| z\|^2_{L^2(Q)}\\
		&&\dis\leq \int_{Q}a_0z^2\ dxdt+\int_{Q}\beta(x)b_0z_xz\ dxdt
		\dis+ \int_{Q}z(h\chi_{\omega}+v^1\chi_{\omega_1}+v^2\chi_{\omega_2})e^{-rt}\ dxdt.
	\end{array}
\end{equation}
We have 
\begin{equation}\label{ze}
\int_{Q}a_0z^2\ dxdt\leq \|a_0\|_{\infty}\| z\|^2_{L^2(Q)}.	
\end{equation}
Using the fact that $\dis |\beta(x)|\leq L\sqrt{a(x)}$ (see \eqref{b1}), we have
\begin{equation}\label{xx}
	\begin{array}{rlll}
	\dis \int_{Q}\beta(x)b_0z_xz\ dxdt&\leq&\dis
	\int_{Q}L\sqrt{a(x)}b_0z_xz\ dxdt\\
	&\leq&\dis \frac{1}{2}L^2\|b_0\|^2_{\infty}\| z\|^2_{L^2(Q)}+\frac{1}{2}\|\sqrt{a(x)}z_x\|^2_{L^2(Q)}.
	\end{array}
\end{equation}
Due to the fact that $\dis e^{-rt}\leq 1,\ \forall t\in [0,T]$, we get
\begin{equation}\label{zea}
	\begin{array}{rlll}
		\dis \int_{Q}z(h\chi_{\omega}+v^1\chi_{\omega_1}+v^2\chi_{\omega_2})e^{-rt}\ dxdt&\leq&\dis
		\int_{Q}z(h\chi_{\omega}+v^1\chi_{\omega_1}+v^2\chi_{\omega_2})\ dxdt\\
		&\leq&\dis \frac{3}{2}\| z\|^2_{L^2(Q)}+\frac{1}{2}\|v^1\|^2_{L^2(\omega_{1,T})}+\frac{1}{2}\|v^2\|^2_{L^2(\omega_{2,T})}+\frac{1}{2}\|h\|^2_{L^2(\omega_T)}.
	\end{array}
\end{equation}
Combining \eqref{ze}-\eqref{zea} with \eqref{estimation1A1}, one obtains

\begin{equation*}
	\begin{array}{rlll}
		&&\dis\frac{1}{2}\|z(T,\cdot)\|^2_{L^2(\Omega)}+\frac{1}{2}\|\sqrt{a(x)}z_x\|^2_{L^2(Q)}+\left(r-\|a_0\|_{\infty}-\frac{1}{2}L^2\|b_0\|^2_{\infty}-\frac{3}{2}\right)\| z\|^2_{L^2(Q)}\\
		&&\dis \leq\frac{1}{2}\|y^0\|^2_{L^2(\Omega)}+\frac{1}{2}\|v^1\|^2_{L^2(\omega_{1,T})}+\frac{1}{2}\|v^2\|^2_{L^2(\omega_{2,T})}+\frac{1}{2}\|h\|^2_{L^2(\omega_T)}.
	\end{array}
\end{equation*}
Taking $r$ such that $\dis r=\|a_0\|_{\infty}+\frac{1}{2}L^2\|b_0\|^2_{\infty}+2$, we obtain
\begin{equation*}
	\begin{array}{rlll}
		\dis\|z(T,\cdot)\|^2_{L^2(\Omega)}+\| z\|^2_{L^2((0,T);H^1_a(\Omega))}
		\dis \leq\|y^0\|^2_{L^2(\Omega)}+\|v^1\|^2_{L^2(\omega_{1,T})}+\|v^2\|^2_{L^2(\omega_{2,T})}+\|h\|^2_{L^2(\omega_T)}.
	\end{array}
\end{equation*}
Since $z=e^{-rt}y$, we deduce the existence  of a constant $C=C(T,\|a_0\|_{\infty}, \|b_0\|_{\infty})>0$ such that the following estimation holds:
\begin{equation*}
	\begin{array}{llllll}
		\dis \|y(T,\cdot)\|^2_{L^2(\Omega)}+\|y\|^2_{L^2((0,T); H^1_a(\Omega))}
		\leq 
		C\left(\|v^1\|^2_{L^2(\omega_{1,T})}+\|v^2\|^2_{L^2(\omega_{2,T})}+\|h\|^2_{L^2(\omega_T)}+\|y^0\|^2_{L^2(\Omega)}\right)
	\end{array}
\end{equation*} 	
and we deduce the inequality \eqref{esty1y2}.

\noindent \textbf{Step 2.} We prove existence by using Theorem \ref{Theolions61}. First of all, it is clear that for any $\phi\in \mathbb{V},$ we have
$$\|\phi\|_{L^2((0,T);H^1_a(\Omega))}\leq \|\phi\|_{\mathbb{V}}.$$
This shows that we have the continuous embedding $\mathbb{V}\hookrightarrow L^2((0,T);H^1_a(\Omega))$.

Now, let $\phi \in \mathbb{V}$ and consider the bilinear form $\mathcal{A}(\cdot,\cdot)$ defined on $L^2((0,T);H^1_a(\Omega))\times \mathbb{V}$ by:
\begin{equation}\label{defCalE}
	\begin{array}{lll}
		\mathcal{A}(y,\phi)&=&\dis -\int_Q y\phi_t \, \dq + \int_Q a(x)y_x\phi_x\,\dq+\int_{Q}a_0 y\phi\, \dq+\int_{Q}\beta(x)b_0 y_x\phi\, \dq.
	\end{array}
\end{equation}
Using Cauchy Schwarz inequality, Remark \ref{rmktrace} and the fact that $\dis |\beta(x)|\leq L\sqrt{a(x)}$, we get that
$$\begin{array}{lllll}
	\dis |\mathcal{A}(y,\phi)|&\leq&\dis  \|y\|_{L^2(Q)}\|\phi_t\|_{L^2(Q)}+ \|\sqrt{a(x)}y_x\|_{L^2(Q)}\|\sqrt{a(x)}\phi_x\|_{L^2(Q)}+\|a_0\|_{\infty}\|y\|_{L^2(Q)}\|\phi\|_{L^2(Q)}\\
	&+& \dis  \|\sqrt{a(x)}y_x\|_{L^2(Q)}\|Lb_0\|_{\infty}\|\phi\|_{L^2(Q)}\\
	&\leq& \dis \sqrt{2}\left[\|\phi_t\|^2_{L^2(Q)}+\|\sqrt{a(x)}\phi_x\|^2_{L^2(Q)}+(\|a_0\|^2_{\infty}+L^2\|b_0\|^2_{\infty})\|\phi\|^2_{L^2(Q)}\right]^{1/2}\|y\|_{L^2((0,T);H^1_a(\Omega))}.
\end{array}
$$
This means that  there is a constant $C=C(\phi,\|a_0\|_{\infty},\|b_0\|_{\infty})>0$  such that
$$|\mathcal{A}(y,\phi)|\leq C\|y\|_{L^2((0,T);H^1_a(\Omega))}. $$
Consequently, for every fixed $\phi\in \mathbb{V},$
the functional  $y\mapsto \mathcal{A}(y,\phi)$ is continuous on $L^2((0,T);H^1_a(\Omega)).$

Next,  we have that for every  $\phi\in\mathbb{V}$,
\begin{equation}\label{xxl}
	\begin{array}{lllll}
		\mathcal{A}(\phi,\phi)&=&\dis -\int_Q \phi\phi_t \, \dq + \int_Q a(x)\phi_x^2\,\dq+\int_{Q}a_0 \phi^2\, \dq+\int_{Q}\beta(x)b_0 \phi_x\phi\, \dq.
	\end{array}
\end{equation}
Due to Assumption \ref{assum}, we get
$$
\int_{Q}a_0 \phi^2\, \dq\geq \alpha\|\phi\|^2_{L^2(Q)}.
$$
Using \eqref{xx}, one have
\begin{equation*}\label{}
	\begin{array}{rlll}
		\dis \int_{Q}\beta(x)b_0\phi_x\phi\ dxdt\geq\dis -\frac{1}{2}L^2\|b_0\|^2_{\infty}\| \phi\|^2_{L^2(Q)}-\frac{1}{2}\|\sqrt{a(x)}\phi_x\|^2_{L^2(Q)}.
	\end{array}
\end{equation*}
Combining the two latter  inequalities with \eqref{xxl} and using again Assumption \ref{assum}, we obtain
$$
\begin{array}{lllll}
	\mathcal{A}(y,\phi)
	&\geq&\dis \frac{1}{2}\|\phi(0,\cdot)\|^2_{L^2(\Omega)}
	+\dis  \frac{1}{2}\|\sqrt{a(x)}\phi_x\|^2_{L^2(Q)}+\left(\alpha-\frac{1}{2}L^2\|b_0\|^2_{\infty}\right)\| \phi\|^2_{L^2(Q)}\\
	&\geq& \dis \min\left\{\frac{1}{2},\left(\alpha-\frac{1}{2}L^2\|b_0\|^2_{\infty}\right)\right\}\|\phi\|^2_{\mathbb{V}}.
\end{array}
$$

Finally, let us consider the linear functional $\mathcal{L}(\cdot):\mathbb{V}\to \R$ defined by
$$\begin{array}{lllll}
	\mathcal{L}(\phi):=\dis \int_{Q} (h\chi_{\omega}+v^1\chi_{\omega_1}+v^2\chi_{\omega_2})\, \phi\; \dq
	+\int_{\Omega} y^0(x)\,\phi(0,x)\ dx.
\end{array}
$$
Then using Remark \ref{rmktrace}, we obtain
$$\begin{array}{lll}
	|\mathcal{L}(\phi)|&\leq&\dis \|h\chi_{\omega}+v^1\chi_{\omega_1}+v^2\chi_{\omega_2}\|_{L^2(Q)}\|\phi\|_{L^2(Q)}+ \|y^0\|_{L^2(\Omega)}\|\phi(0,\cdot)\|_{L^2(\Omega)}\\
	&\leq&\dis (\|h\chi_{\omega}+v^1\chi_{\omega_1}+v^2\chi_{\omega_2}\|_{L^2(Q)}+ \|y^0\|_{L^2(\Omega)})\|\phi\|_{\mathbb{V}}\\
	&\leq & C\|\phi\|_{\mathbb{V}},
\end{array}
$$
where $C= C(T,h,v^1,v^2)>0$.
Therefore,  $\mathcal{L}(\cdot)$ is continuous on $\mathbb{V}$. Thus,  it follows from Theorem \ref{Theolions61} that there exists $y\in L^2((0,T);H^1_a(\Omega))$ such that
\begin{equation}\label{formvar}
	\mathcal{A}(y,\phi)= \mathcal{L}(\phi),\quad \forall \varphi \in \mathbb{V}.
\end{equation}
We have shown that the system \eqref{ywell} has a solution $y\in L^2((0,T);H^1_a(\Omega))$ in the sense of Definition \ref{weaksolution}. In addition, using the first equation of \eqref{ywell}, we deduce that $y_t\in L^2((0,T);(H^{1}_a(\Omega))^\prime)$. So $y \in W_a(0,T)$ and using Remark \ref{remwta},  it follows that $y\in \mathbb{H}$.\\

	\noindent \textbf{Step 3.} We prove uniqueness.
Assume that there exist $y_1$ and $y_2$ solutions to \eqref{ywell} with the same right hand side $h,\, v^1, \, v^2$ and initial datum $y^0$.  Set $z:=e^{-rt}(y_1-y_2)$. Then $z$ satisfies
\begin{equation}\label{u0}
	\left\{
	\begin{array}{rllll}
		\dis z_t-(a(x)z_x)_x+a_0z+\beta(x)b_0z_x+r z &=&0 \qquad &\mbox{in}& Q,\\
		\dis z(t,0)=z(t,1)&=&0& \mbox{on}& (0,T), \\
		\dis z(0,\cdot)&=&0&\mbox{in}&\Omega.
	\end{array}
	\right.
\end{equation}
So, if we  multiply the first equation in \eqref{u0} by $z$, and integrate by parts over $Q$, we obtain
\begin{equation*}
	\begin{array}{rllll}
		\dis \dis\frac{1}{2}\|z(T,\cdot)\|^2_{L^2(\Omega)}+\frac{1}{2}\|\sqrt{a(x)}z_x\|^2_{L^2(Q)}+\left(r-\|a_0\|_{\infty}-\frac{1}{2}L^2\|b_0\|^2_{\infty}\right)\| z\|^2_{L^2(Q)}\leq 0.
	\end{array}
\end{equation*}
Choosing $\dis r=\|a_0\|_{\infty}+\frac{1}{2}L^2\|b_0\|^2_{\infty}+\frac{1}{2}$ in this latter inequality, we deduce that 
\begin{equation*}
	\begin{array}{rllll}
		\|z\|^2_{L^2((0,T);H^1_a(\Omega))}\leq 0.
	\end{array}
\end{equation*}
Hence $z=0$ in $Q$ and then, $y_1=y_2$	in $Q$ and we have shown uniqueness.\\
This complete the proof.
\end{proof}

\subsection{Proof of Theorem \ref{exissemi}}
\begin{proof}$ $
	
	We divide the proof into two steps.\\
	\noindent \textbf{Step 1.}
We prove the existence of the weak solution to the problem \eqref{eq} by using the Schauder fixed point theorem. We proceed as in Theorem $2.1$ \cite{xu2020}.

For any $z\in L^2((0,T);H^1_a(\Omega))$, it holds that
 \begin{equation*}
	\begin{array}{rllll}
		\dis F(z,z_x)-F(0,0)&=&\dis \int_{0}^{1}\frac{\partial}{\partial r}F(rz,rz_x)\ dr\\
		&=&\dis \int_{0}^{1}zD_1F(rz,rz_x)\ dr+\int_{0}^{1}z_xD_2F(rz,rz_x)\ dr.
	\end{array}
\end{equation*}	
Define
$$
d_1(z)=\dis \int_{0}^{1}D_1F(rz,rz_x)\ dr\ \mbox{and}\ d_2(z)=\dis \frac{1}{\beta(x)}\int_{0}^{1}D_2F(rz,rz_x)\ dr.
$$		
Then
$$
 F(z,z_x)-F(0,0)=d_1(z)z+\beta(x)d_2(z)z_x.
$$
Moreover, using the point $(H_5)$ of Assumption \ref{assum}, one get
\begin{equation}
\|d_1\|_{\infty}\leq M_2,\ \ \|d_2\|_{\infty}\leq M_2.
\end{equation}
Now, the system \eqref{eq} becomes
 \begin{equation}\label{ywella}
	\left\{
	\begin{array}{rllll}
		\dis y_{t}-\left(a(x)y_{x}\right)_{x}+d_1(z)y+\beta(x)d_2(z)y_x  &=&\dis h\chi_{\omega}+v^1\chi_{\omega_1}+v^2\chi_{\omega_2}& \mbox{in}& Q,\\
		\dis y(t,0)=y(t,1)&=&0& \mbox{on}& (0,T), \\
		\dis y(0,\cdot)&=&y^0&\mbox{in}&\Omega.
	\end{array}
	\right.
\end{equation}
It follows from Proposition \ref{exis} that the problem \eqref{ywella} admits a unique weak solution $y\in \mathbb{H}$. Define the mapping $\Lambda:L^2((0,T);H^1_a(\Omega))\longrightarrow L^2((0,T);H^1_a(\Omega))$ as follows:
$$
\Lambda(z)=y,\ \ z\in L^2((0,T);H^1_a(\Omega)),
$$	
where $y$ is the weak solution to problem \eqref{ywella}. Proceeding as in \cite{xu2020}, we can prove that the operator $\Lambda$ is continuous and compact. Moreover the range of $\Lambda$ is bounded. The operator $\Lambda$ satisfies the hypotheses of the Schauder fixed point theorem. Therefore, $\Lambda$ admits a fixed point $y\in L^2((0,T);H^1_a(\Omega))$ such that $y=\Lambda(y)\in \mathbb{H}$ is a weak solution to the problem \eqref{eq}.	

\noindent \textbf{Step 2.}
 Now, we want to prove the uniqueness of the weak solution. Assume that $\tilde{y}$ and $\bar{y}$ are two weak solutions to the system \eqref{eq} and 
 set
$$
W=\tilde{y}-\bar{y}.
$$
Observe that 
$$
F(\tilde{y}, \tilde{y}_x)-F(\bar{y}, \bar{y}_x)=\theta_1(t,x)W+\theta_2(t,x)W_x,\ (t,x)\in Q,
$$
where
$$
\theta_1(t,x)=\int_{0}^{1}D_1F(r\tilde{y}+(1-r)\bar{y},r\tilde{y}_x+(1-r)\bar{y}_x)\ dr,
$$
$$
\theta_2(t,x)=\frac{1}{\beta(x)}\int_{0}^{1}D_2F(r\tilde{y}+(1-r)\bar{y},r\tilde{y}_x+(1-r)\bar{y}_x)\ dr.
$$
Then $W$ is the solution to the following problem
\begin{equation*}\label{}
	\left\{
	\begin{array}{rllll}
		\dis W_{t}-\left(a(x)W_{x}\right)_{x}+\theta_1(t,x)W+\beta(x)\theta_2(t,x)W_x  &=&0& \mbox{in}& Q,\\
		\dis W(t,0)=W(t,1)&=&0& \mbox{on}& (0,T), \\
		\dis W(0,\cdot)&=&0&\mbox{in}&\Omega.
	\end{array}
	\right.
\end{equation*}
Thanks to $(H_5)$ and $(H_6)$ of Assumption \ref{assum}, there exists a positive constant $M$ such that
\begin{equation*}
	\|\theta_1\|_{\infty},\ \|\theta_2\|_{\infty} \leq M.
\end{equation*}
It follows from Proposition \ref{exis} that
$$
W=0\ \mbox{in}\ Q.
$$
Then 
$$
\tilde{y}=\bar{y}\ \mbox{in}\ Q.
$$
The proof is complete.	
\end{proof}

\section*{Disclosure statement}
The authors report there are no competing interests to declare.

\bibliographystyle{plain}
\bibliography{references}

\end{document}